\documentclass[a4paper,11pt]{amsart}
\usepackage{amssymb}
\usepackage{amsmath}
\newtheorem{lemma}{Lemma}
\newtheorem{prop}{Proposition}
\newtheorem{thm}{Theorem}

\newtheorem{cor}{Corollary}
\theoremstyle{definition}
\newtheorem{defn}{Definition}
\theoremstyle{remark}
\newtheorem{rem}{Remark}

\newtheorem{conj}{Conjecture}
\newtheorem{problem}{Problem}
\newcounter{numl}
\newcommand{\labelnuml}{\textup{(\roman{numl})}}
\newenvironment{numlist}{\begin{list}{\labelnuml}%
{\usecounter{numl}\setlength{\leftmargin}{0pt}%
\setlength{\itemindent}{2\parindent}%
\setlength{\itemsep}{\smallskipamount}\def
\makelabel ##1{\hss \llap {\upshape ##1}}}}{\end{list}}
\newenvironment{bulletlist}{\begin{list}{\labelitemi}%
{\setlength{\leftmargin}{\parindent}\def
\makelabel ##1{\hss \llap {\upshape ##1}}}}{\end{list}}

\DeclareSymbolFont{script}{U}{eus}{m}{n}
\DeclareSymbolFontAlphabet{\mathscr}{script}
\DeclareMathSymbol{\Wedge}{0}{script}{"5E}
\DeclareMathAlphabet{\mathrmsl}{OT1}{cmr}{m}{sl}
\newcommand{\R}{{\mathbb R}}
\newcommand{\C}{{\mathbb C}}
\newcommand{\Z}{{\mathbb Z}}
\newcommand{\N}{{\mathbb N}}

\newcommand{\T}{{\mathbb T}}

\newcommand{\sub}{\subseteq}

\newcommand{\trace}{\mathop{\mathrm{tr}}\nolimits}

\newcommand{\tor}{{\mathfrak t}}

\newcommand{\Fa}{F}

\renewcommand{\d}{{\mathrmsl d}}

\newcommand{\Hess}{\mathop{\mathrm{Hess}}}

\newcommand{\ip}[1]{\langle #1 \rangle}

\newcommand{\J}{\boldsymbol J}
\newcommand{\Om}{\boldsymbol\Omega}
\newcommand{\ang}{\boldsymbol t}

\newcommand{\Proj}{\mathrm P}

\newcommand{\al}{\alpha}
\newcommand{\be}{\beta}

\newcommand{\1}{0}
\newcommand{\2}{\infty}

\begin{document}

\author[V. Apostolov]{Vestislav Apostolov}
\address{Vestislav Apostolov \\ D{\'e}partement de Math{\'e}matiques\\
UQAM\\ C.P. 8888 \\ Succursale Centre-ville \\ Montr{\'e}al (Qu{\'e}bec) \\
H3C 3P8 \\ Canada}
\email{apostolov.vestislav@uqam.ca}
\author[G. Maschler]{Gideon Maschler}
\address{Gideon Maschler \\ Department of Mathematics and Computer Science\\
Clark University\\
Worcester\\ Massachusetts\\ 01610\\ U.S.A.}
\email{gmaschler@clarku.edu}
\title{Conformally K\"ahler,  Einstein--Maxwell Geometry}

\date{\today}
\begin{abstract} On a given compact complex manifold or orbifold $(M,J)$, we study the existence of Hermitian metrics $\tilde g$ in the conformal classes of  K\"ahler metrics on $(M,J)$,  such that  the Ricci tensor of $\tilde g$ is  of type $(1,1)$ with respect to the complex structure,  and the scalar curvature of $\tilde g$ is constant.   In real  dimension $4$,  such Hermitian metrics provide a Riemannian counter-part of the Einstein--Maxwell equations in general relativity, and have been recently studied in \cite{ambitoric1, LeB0, LeB,KTF}. We show how the existence problem of such Hermitian metrics (which we call in any dimension  {\it conformally K\"ahler, Einstein--Maxwell} metrics) fits into a formal momentum map interpretation, analogous to results by  Donaldson and Fujiki~\cite{donaldson, fujiki} in the constant scalar curvature K\"ahler case. This leads to a suitable notion of a Futaki invariant which provides an obstruction to the existence of conformally K\"ahler,  Einstein--Maxwell  metrics invariant
under a certain group of automorphisms which are associated to a given K\"ahler class,  a real holomorphic vector field on $(M,J)$, and a positive normalization constant. Specializing to the toric case, we further define a suitable notion of $K$-polystability and show it provides a (stronger) necessary condition for the existence of toric, conformally K\"ahler,  Einstein--Maxwell  metrics.  We use the methods of \cite{ambitoric2} to show that on a compact symplectic toric $4$-orbifold with second Betti number equal to $2$, $K$-polystability is also a sufficient condition for the existence of (toric) conformally K\"ahler,  Einstein--Maxwell metrics,  and the latter are explicitly described as ambitoric in the sense of \cite{ambitoric1}. As an application,  we exhibit many new examples of conformally K\"ahler,  Einstein--Maxwell metrics defined on compact $4$-orbifolds,  and obtain a uniqueness result for the construction in \cite{LeB0}. 

\end{abstract}

\maketitle

\section{Introduction}
\subsection{Conformally K\"ahler, Einstein-Maxwell metrics}\label{s:preliminaries} In this paper, we study a special class of  (non-K\"ahler in general) Hermitian metrics  $\tilde g$, defined on a compact complex K\"ahler  manifold (or orbifold)  $(M,J)$ of real dimension $2m \ge 4$ as follows.
\begin{defn} Let $\tilde g$ be a Hermitian metric on $(M,J)$ for which there exists a positive smooth function $f$ such that $g= f^2 \tilde g$ is a K\"ahler metric, and which satisfies the following curvature conditions:
\begin{equation}\label{ric}
{\rm Ric}^{\tilde g}(J \cdot, J \cdot ) = {\rm Ric}^{\tilde g}(\cdot, \cdot),
\end{equation}
\begin{equation}\label{scal}
s_{\tilde g} = {\rm const},
\end{equation}
where ${\rm Ric}^{\tilde g}$ and $s_{\tilde g}$ denote the Ricci tensor and the scalar curvature of $\tilde g$.
 We shall refer to such a conformally K\"ahler,  Hermitian metric  as {\it conformally K\"ahler, Einstein--Maxwell}  metric on $(M,J)$.
\end{defn}
Note that, the Ricci tensor ${\rm Ric}^g$ of the K\"ahler metric $g= f^2 \tilde g$ also satisfies \eqref{ric}, and the Ricci tensors of $\tilde{g}$ and $g$ are related by (see e.g. \cite[1.161]{besse})
\begin{equation*}
{\rm Ric}^{\tilde g} = {\rm Ric}^g + \frac{2m-2}{f} D^g df  + h g,
\end{equation*}
where $D^g$ denotes the Levi--Civita connection of $g$ and $h$ is a smooth function not given here explicitly.  It is thus easily seen that condition \eqref{ric} is equivalent to the fact that the vector field $K= J {\rm grad}_g f$ is Killing for both $g$ and $\tilde g$, whereas condition \eqref{scal} reads as
\begin{equation}\label{yamabe}
s_{\tilde g}=2\Big(\frac{2m-1}{m-1}\Big) f^{m+1} \Delta_{g}\Big(\frac{1}{f}\Big)^{m-1} + s_{g} f^2 = {\rm const},
\end{equation}
where $\Delta_g$ is the Riemannian Laplacian of $g$ and $s_g$ is the scalar curvature of $g$. In other words, conformally K\"ahler,  Hermitian metrics $\tilde g$ on $(M,J)$  satisfying \eqref{ric}-\eqref{scal} are in one-to-one correspondence with K\"ahler metrics $g$ on $(M, J)$  which admit a  Killing vector field $K$ with a positive Killing potential $f$,  satisfying \eqref{yamabe}.

\smallskip
In order to explain the relation with the Einstein--Maxwell equations, recall that any Hermitian metric $(\tilde g, J)$  such that $f^2\tilde g$ is K\"ahler satisfies the relation
\begin{equation}\label{basic}
\tilde{g}(D^{\tilde g}_X J \cdot, \cdot)) = \frac{1}{2}\Big(X^\flat \wedge J\theta + JX^{\flat} \wedge \theta\Big),
\end{equation}
where  $\theta = -2 d \log f$,  $D^{\tilde g}$ stands for the Levi--Civita connection of $\tilde g$ and $\flat$ (resp. $\sharp$) denotes the isomorphism $\tilde g: T_pM \to T^*_pM$
(resp. $\tilde g^{-1} : T^*_p M \to T_pM$). This is essentially \cite[Ch.~IX, Prop.~4.2]{KN}, taking in mind that  in our case $dF= \theta \wedge F$,  where $F(X,Y):=\tilde g(JX, Y)$ denotes the fundamental $2$-form of $(\tilde g, J)$).  Letting ${\rm Ric}_0^{\tilde g}$ denote the trace-free part of ${\rm Ric}^{\tilde g}$ and $\rho_0^{\tilde g}(\cdot, \cdot)= {\rm Ric}^{\tilde{g}}_0(J \cdot, \cdot)$ the corresponding primitive $(1,1)$-form, it follows that
\begin{equation}\label{second}
\begin{split}
(\delta^{\tilde g} \rho_0^{\tilde g})(X) =&\sum_{i=1}^{2m}\Big({\rm Ric}^{\tilde g}_0(e_i, (D^{\tilde g}_{e_i} J)(X))  + \big((D^{\tilde g}_{e_i} {\rm Ric}_0^{\tilde g}) (e_i, JX)\big)\Big)\\
                            = &-\frac{1}{2} \sum_{i=1}^{2m} \Big(({e_i}^\flat \wedge J\theta+ J{e_i}^{\flat} \wedge \theta \big)\big({\rm Ric}_0^{\tilde g}(e_i)^{\sharp}, X) \Big)\\
                            &- (\delta^{\tilde g} {\rm Ric}^{\tilde g}_0) (JX)\\
                            = & \rho^{\tilde g}_0 (\theta^{\sharp}, X) + \frac{(m-1)}{2m} (Jds_{\tilde g})(X),
 \end{split}
 \end{equation}
where $\{e_1, \ldots, e_{2m}\}$ is any $\tilde g$-orthonormal frame, $\delta^{\tilde g}$ is the formal adjoint of $D^{\tilde g}$, and we have used \eqref{basic} and the Ricci identity $\delta^{\tilde g}  {\rm Ric}^{\tilde g}_0 =   \frac{m-1}{2m} ds_{\tilde g}^{\sharp}$.        It follows that the $(1,1)$-form $\frac{1}{f^2} \rho_0^{\tilde g}$ is co-closed iff $s_{\tilde g}$ is constant. In other words, for any conformally-K\"ahler, Einstein--Maxwell metric $\tilde g$, the trace-free Ricci endomorphism $r^{\tilde g}_0$ can be written as
\begin{equation}\label{EM}
r^{\tilde g}_0 = -(\omega^{\sharp})\circ (\frac{1}{f^2} \rho_0^{\tilde g})^{\sharp},
\end{equation}
where $\omega(\cdot, \cdot)= g(J\cdot, \cdot)$ is the {\it closed} K\"ahler form of the K\"ahler metric $g$, while $(\frac{1}{f^2} \rho_0^{\tilde g})$ is a {\it co-closed}  $(1,1)$-form which is orthogonal to $\omega$, and the superscript $\sharp$ denotes the skew-symmetric endomorphisms corresponding to these forms via the metric $\tilde g$.

When $m=2$,  both $\omega$ and    $(\frac{1}{f^2} \rho^{\tilde g}_0)$ are {\it harmonic} with respect to $\tilde g$, as they are, respectively, closed self-dual and co-closed anti-self-dual  $2$-forms. Thus, in this case,  $\tilde g$ is a metric of Riemannian signature satisfying the Einstein--Maxwell equations in General Relativity. We also notice that in this case,  the Riemannian Goldberg--Sachs theorem (see e.g. \cite{AG}) implies that a Hermitian metric $\tilde g$ which satisfies
\eqref{ric} and \eqref{scal} is automatically  conformally K\"ahler as soon as $(M,J)$ admits a K\"ahler metric (i.e. the first Betti number of $M$ is even).

\subsection{Motivation and examples}

There are two special cases of conformally K\"ahler, Einstein--Maxwell metrics,  which are well-studied.

\smallskip
The first one is the case when the conformal factor $f$ is constant,  i.e. $\tilde g$ is a constant scalar curvature K\"ahler metric (cscK for short). The theory of cscK metrics represents  a most active area of current research. While this case will be considered trivial from the point of view of the theory developed in this paper (as in this case  it will add no new results), the theory of cscK metrics will play a pivotal role in motivating the results in the case when $f$ is not constant.

\smallskip
The second case of special interest is when $\tilde g$ is {\it Einstein}, i.e. ${\rm Ric}^{\tilde g}= \frac{s_{\tilde g}}{2m} \tilde{g}$. When $m=2$,  smooth compact complex surfaces admitting conformally K\"ahler, Einsten metrics have been classified by Chen--LeBrun--Weber~\cite{CLW}. Apart from the K\"ahler--Einstein case (i.e. when $f=const.$) which is classified in \cite{tian0}, there are only two such Hermitian surfaces, namely the first Hirzebruch surface $\mathbb{F}_1$ with the Page metric~\cite{page} on it, and the blow-up of $\mathbb{F}_1$ at one point, for which existence of the metric was shown in \cite{CLW}. In both cases, the resulting Einstein metrics are toric, i.e. contain a $2$-dimensional torus in their isometry groups. Furthermore, it was shown in \cite{ambitoric2} that there exist continuously countable families of toric compact Einstein, conformally K\"ahler $4$-orbifolds.

When $m>2$, the situation is more rigid, as shown by Derdzinski--Maschler \cite{DM0,DM}: the only compact smooth complex manifolds  admitting conformally K\"ahler, non-K\"ahler,  Einsten metrics are certain $\C P^1$-bundles over a Fano, K\"ahler--Einstein manifold with a metric found by L.~B\'erard-Bergery~\cite{B-B}. However, it appears likely that the methods of \cite{DM} allow to construct more orbifold examples.

\smallskip
It  was recently observed \cite{LeB0,LeB} that many more examples of conformally K\"ahler, Einstein--Maxwell metrics exist on {\it toric}  compact complex surfaces  $(M,J)$ with second Betti number $b_2(M)=2$ (i.e. $\C P^1 \times \C P^1$ and the Hirzebruch surfaces $\mathbb{F}_k, k\ge 1$), and that these examples share a remarkable resemblance with the theory of extremal K\"ahler metrics pioneered by Calabi~\cite{calabi}, and intensively developed since then. Further examples on ruled complex surfaces of higher genus are found in the recent work \cite{KTF}. A large family of local toric examples also appears in \cite{ambitoric1}.

\subsection{Outline of the main results}

First, we recast the problem of finding (non-K\"ahler) conformally-K\"ahler,  Einstein--Maxwell metrics within the framework of moment maps for the action of a suitable infinite dimensional subgroup of hamiltonian diffeomorphisms on the Fr\'echet space of compatible complex structures on a given compact symplectic manifold or orbifold $(M,\omega)$, with respect to a suitably defined symplectic structure. This is the content of Theorem~\ref{thm:moment-map-setting}, which can be viewed as an extension of the results by Fujiki~\cite{fujiki} and Donaldson~\cite{donaldson} obtained in the cscK case to the more general conformally--K\"ahler, Einstein--Maxwell case. We emphasize that this construction depends upon fixing in advance a hamiltonian vector field $K$ within the Lie algebra of a compact subgroup $G$ of hamiltonian transformations of $(M, \omega)$, as well as a positive hamiltonian $f$ for $K$.

As a direct consequence of the formal GIT picture provided by Theorem~\ref{thm:moment-map-setting}, we show in Corollary~\ref{c:symplectic-futaki} that there exists a Futaki invariant which is an obstruction to the existence of  Einstein--Maxwell  metrics with conformal factor $f$, in the conformal classes of  $G$-invariant, $\omega$-compatible K\"ahler metrics on a given compact symplectic manifold  (or orbifold) $(M,\omega)$.

As another application of Theorem~\ref{thm:moment-map-setting}, we consider the problem of the existence of conformally-K\"ahler, Einstein--Maxwell metrics within the classical framework of K\"ahler geometry,  pioneered by the work of Calabi~\cite{calabi}, in which the complex structure $J$ is fixed, and the K\"ahler metric $g$ varies within  a given K\"ahler class on $(M,J)$. Corollary~\ref{c:complex-futaki} provides  the relevant  notion of Futaki invariant in this setting.

\smallskip
In Section~\ref{s:toric},  we specialize to the case when $(M, \omega, \T)$ is a compact toric symplectic manifold or orbifold, and $K \in {\rm Lie}(\T)$ is a vector field with a positive hamiltonian $f$. We demonstrate how the existence problem for toric conformally-K\"ahler, Einstein--Maxwell metrics with conformal factor $f$ fits within the general framework of toric cscK metrics, developed by S. Donaldson~\cite{Do-02}. This leads both to a more subtle obstruction, the so-called $K$-polystability (Theorem~\ref{toric-futaki} and Corollary~\ref{c:K-stability}) and a uniqueness result (Theorem~\ref{uniqueness}). The main conjecture from \cite{Do-02} is then that a solution exists if and only if the corresponding Delzant polytope of $(M,\omega, \T)$  is $K$-polystable (Conjecture~\ref{c:donaldson}).

\smallskip
As an application of the theory,  in Section~\ref{s:examples} we specialize to the case of compact symplectic toric $4$-orbifolds whose Delzant polytope is either a triangle (i.e., up to an orbifold cover $(M,\omega)$ is a weighted projective plane) or a quadrilateral (equivalently, a compact toric orbifold with second Betti number $b_2=2$).

In the first case,  we prove in Theorem~\ref{wpp-classification} that  any conformally-K\"ahler, Einstein--Maxwell metric must be Einstein and, therefore,  conformal to the Bochner--flat K\"ahler metric on such a space, see  \cite{Bryant, DG} for a classification.

In the second case, we demonstrate in Theorem~\ref{ambitoric-EM-classification} how a straightforward modification of the analysis in \cite{ambitoric2} confirms the main Conjecture~\ref{c:donaldson} in this special toric case.  As a bi-product, we establish an interesting bijective correspondence between compact toric $4$-orbifolds with $b_2=2$,  admitting conformally K\"ahler,  Einstein--Maxwell metrics on one side,  and compact toric $4$-orbifolds with $b_2=2$ admitting  an extremal K\"ahler metric with positive scalar curvature on the other (Proposition~\ref{duality}). This leads to many new examples (Corollary~\ref{c:existence}) which can be considered as orbifold compactifications of the Riemannian analogues of the Pleba\'nski--Demia\'nski~\cite{PD} Lorentzian Einstein--Maxwell metrics.

\smallskip
Various related topics are discussed in Section \ref{varia}.

The first one is  the definition of a quantized version of the new Futaki invariant
on a polarized variety,  which we discuss in Section~\ref{s:quantized}. This leads to a suitable notion of `$K$-stability' for conformally-K\"ahler, Einstein--Maxwell metrics associated to a polarization.  While we do not explore in this paper the relation between this notion of $K$-stability and the existence of conformally-K\"ahler, Einstein--Maxwell metrics, we note that in the toric case, this suggested notion reduces to the one introduced in Section~\ref{s:toric} when one restricts to special (toric) degenarations.

In Section~\ref{s:conformally-einstein},  we characterize conformally K\"ahler, Einstein--Maxwell metrics on connected four manifolds
which are in fact Einstein, showing that the corresponding K\"ahler metric must be extremal and the conformal factor must be a multiple of its scalar curvature.

In the final Section~\ref{s:computing-futaki},  we present two uniqueness results,  involving explicit calculations of
the Futaki invariant. The first one (Proposition~\ref{wpp-futaki}) shows that the Futaki invariant of a toric K\"ahler metric  with respect to a positive Killing potential on a weighted projective plane $\C P^2_{a_0, a_1,a_2}$ vanishes if and only if $\C P^2_{a_0, a_1, a_2}$ admits an Einstein, conformally-K\"ahler  metric,  in which case, as mentioned above, the function $f$ is determined uniquely up to scale.  The second result (Corollary~\ref{thm:classification-product}) shows that a conformally
K\"ahler, Einstein--Maxwell metric on $\mathbb{CP}^1\times\mathbb{CP}^1$,  which  has a $2$-dimensional torus worth of isometries,  must be isometric to one of the examples found by LeBrun in \cite{LeB0}. Symbolic computational software calculations of the Futaki invariant are used in the derivation of this result.

\section*{Acknowledgement} The first author was supported in part by an NSERC Discovery Grant and is grateful to the Institute of Mathematics and Informatics of the Bulgarian Academy of Sciences where a part of this project was realized. He is also grateful to  the Simons Center for Geometry and Physics, Stony Brook, for a workshop invitation in 2015. The authors thank  Song Sun and G\'abor Sz\'ekelyhidi for their valuable suggestions concerning Section~\ref{s:quantized}, as well as Simon Donladson and  Claude LeBrun
for helpful discussions and comments. They also thank Christina T{\o}nnesen-Friedman for exchanges involving the paper \cite{KTF}.
Finally, the second author acknowledges the Centre Interuniversitaire de Recherche en G\'eom\'etrie Et Topologie (CIRGET) in
Universit\'e du Qu\'ebec \`a Montr\'eal for its support and hospitality during the visit in which this work was initiated.

\section{A Futaki invariant}\label{s:futaki}

\subsection{The momentum map setting}\label{s:symplectic-futaki} In this section we adapt the arguments from \cite{donaldson,fujiki}. Let $\tilde g= \frac{1}{f^2} g$ be a conformally K\"ahler,  Einstein--Maxwell metric on a compact complex manifold $(M,J)$ of complex dimension $m \ge 2$.  As mentioned in Section~\ref{s:preliminaries}, the vector field $K=J {\rm grad}_{g} f$ is Killing for both $g$ and $\tilde g$ and has zeroes. It follows that $K \in {\rm Lie}(\tilde G)$,  where {\rm Lie} denotes a Lie algebra and $\tilde G= {\rm Isom}(M,\tilde g) \cap {\rm Aut}_{r}(M,J)$ is intersection of the isometry group of $\tilde g$ with the group ${\rm Aut}_r(M,J)$ of reduced automorphisms of $(M,J)$, see e.g. \cite{gauduchon-book} for a definition. Notice that the K\"ahler metric $g$ is also $\tilde G$-invariant,  and $K$ belongs to the center of ${\rm Lie}(\tilde G)$. Indeed,  $\tilde G$ acts by conformal transformations of $g$ and preserves the complex structure $J$. As the only K\"ahler metrics in the conformal class $[g]$ are homotheties of $g$, and $\tilde G$ acts trivially in cohomology (being a subgroup of ${\rm Aut}_r(M,J)$), it follows that $\tilde G$ preserves the K\"ahler form $\omega=g(J\cdot, \cdot)$ (as it is harmonic), hence also $g$.  As both $g$ and $\tilde g$ are $\tilde G$-invariant, the conformal factor $\frac{1}{f^2}$ must be $\tilde G$ invariant too, i.e. $K$ must be in the center of ${\rm Lie}(\tilde G)$.  For practical purposes, it will be more convenient to work with a maximal torus $G\subset \tilde G$ in the connected component of the identity of $\tilde G$, but this assumption is not essential in what follows, as long as $K$ belongs to the Lie algebra $\mathfrak{g}$ of the subgroup $G$. Note that being central, $K \in \mathfrak{g}:={\rm Lie}(G)$ if $G$ is a maximal sub-torus of $\tilde G$. We shall further fix a Killing potential $f$ of $K$, which is, equivalently, a hamiltonian function associated to $K$, i.e. $\imath_{K} \omega = -df$.

Denote by $\mathcal{C}_{\omega}^G$ the space of all $\omega$-compatible $G$-invariant complex structures on $(M, \omega)$ and consider the natural action on $\mathcal{C}_{\omega}^G$ of the group ${\rm Ham}^{G}(M,\omega)$ of $G$-equivariant hamiltonian transformations of $(M,\omega)$. We shall  identify its Lie algebra  with the space $(C^{\infty}_0(M))^G$ of $G$-invariant smooth functions on $M$ with zero mean with respect to $v_{\omega}= \omega^m/m!$ (endowed with the Poisson bracket).

The space $\mathcal{C}_{\omega}^G$ carries  a formal Fr\'echet K\"ahler structure, $(\J, {\Om}^f)$, defined  by
\begin{equation*}
{\J}_{J} (\dot J) = J \dot J, \, \ \  {\Om}^f_{J}(\dot{J}_1, \dot{J}_2) = \frac{1}{2} \int_M {\rm tr} \Big(J \dot{J}_1 \dot{J}_2\Big) \frac{1}{f^{2m-1}} v_{\omega},
\end{equation*}
where the tangent space of $\mathcal{C}_{\omega}^G$ at $J$ is identified to be the Fr\'echet space of smooth sections $\dot J$ of ${\rm End}(TM)$
satisfying
\begin{equation*}
\dot{J} J + J \dot{J} =0,  \, \, \   \omega(\dot{J} \cdot , \cdot) + \omega(\cdot, \dot{J} \cdot) =0.
\end{equation*}
Note that the formal complex structure ${\J}$ is the same as the one in \cite{donaldson,fujiki}, whereas we have used the fixed positive hamiltonian  $f$ in order to modify the formal symplectic form $\Om^f$ (which stays closed, as can easily be checked).

In the following we shall use the notation $g_J$ to denote the K\"ahler metric corresponding to a given $G$-invariant
$\omega$-compatible complex structure $J$: $g_J(\cdot,\cdot)=\omega(\cdot,J\cdot )$. The conformally related metric $\frac{1}{f^2} g_J$
will be denoted $g_{J,f}$. Objects associated to these metrics, such as their scalar curvatures, will be indexed similarly.
\begin{thm}\label{thm:moment-map-setting} The action of ${\rm Ham}^{G}(M,\omega)$ on $(\mathcal{C}_{\omega}^G, \J, \Om^f)$ is hamiltonian with a momentum map given by the $\langle \cdot, \cdot \rangle_f$-dual of the scalar curvature $s_{J,f}$ of the Hermitian metric $g_{J,f} = \frac{1}{f^2} g_J$, where
$\langle \varphi, \psi \rangle_f= \int_M \varphi \psi \frac{1}{f^{2m+1}} v_{\omega}$ is an ${\rm ad}$-invariant scalar product defined on the Lie algebra $(C^{\infty}_0(M))^G$ of ${\rm Ham}(M,\omega)^G$.
\end{thm}
\begin{proof} Recall that the scalar curvature $s_{J,f}$ of  $g_{J,f} = \frac{1}{f^2} g_J$ is related to the scalar curvature $s_J$ of $g_J$ by \eqref{yamabe}, i.e.
\begin{equation}\label{s-symplectic}
s_{J,f} =2\Big(\frac{2m-1}{m-1}\Big) f^{m+1} \Delta_{J}\Big(\frac{1}{f}\Big)^{m-1} + s_{J} f^2,
\end{equation}
where $\Delta_J$ denotes  the Laplace operator of  $g_{J}$.  Integrating by parts, it follows  that for any smooth function $h$
\begin{equation}\label{rough-futaki}
\begin{split}
\int_M &s_{J,f} \Big(\frac{h}{f^{2m+1}}\Big) v_{\omega} =  2m(2m-1) \int_M \frac{h}{f^{2m+1}}g_J(df, df) v_{\omega}\\
       & -2(2m-1) \int_M \frac{1}{f^{2m}}g_J(df, dh) v_{\omega}  + \int_M s_J \Big( \frac{h}{f^{2m-1}}\Big) v_{\omega} \\
       =&\ \ \  2m(2m-1) \int_M \frac{h}{f^{2m+1}}g_J\Big(\omega^{-1}(df),  \omega^{-1}(df)\Big) v_{\omega}\\
       & -2(2m-1) \int_M \frac{1}{f^{2m}}g_J\Big(\omega^{-1}(df), \omega^{-1}(dh)\Big) v_{\omega}  \\
       &+ \int_M s_J \Big( \frac{h}{f^{2m-1}}\Big) v_{\omega},
\end{split}
\end{equation}
where $\omega^{-1} = -Jg^{-1} : T^*_pM \to T_pM$ denotes the symplectic isomorphism between $1$-forms  and  vectors.

For a curve $J(t) \in \mathcal{C}^G_{\omega}$, denote by $\dot{J}= \frac{d}{dt} J(t)$ its first variation. As $J(t)$ are $\omega$-compatible almost-complex structures, we have $\dot{J} J = - J\dot J$ and $\omega(\dot J \cdot , \cdot ) = g_{J}(J \dot J, \cdot, \cdot)= -\omega(\cdot, \dot{J} \cdot)=g_J(\cdot , J \dot J \cdot)$. It follows that $\dot{g}_J(\cdot, \cdot) = -g_J (\cdot,  J \dot J \cdot)= g_J(\cdot, \dot{J} J \cdot)$. It is well-known, see e.g. \cite{gauduchon-book}, that with respect to the K\"ahler metric $g_J$,
$\dot{s}_{J} =- \delta J \delta \dot{J}$. Thus, for any path $J(t) \in \mathcal{C}^{G}_{\omega}$ and a smooth function $h$, we have
\begin{equation}\label{long-computation}
\begin{aligned}
\frac{d}{dt} &\int_M  (s_{J(t),f}) \frac{h}{f^{2m+1}} v_{\omega}  = 2m(2m-1) \int_M \frac{h}{f^{2m+1}}\dot{g}_J\Big(\omega^{-1}(df), \omega^{-1}(df)\Big) v_{\omega}\\
        & -2(2m-1) \int_M \frac{1}{f^{2m}}\dot{g}_J\Big(\omega^{-1}(df), \omega^{-1}(dh)\Big) v_{\omega}+\int_M \dot{s}_J \Big( \frac{h}{f^{2m-1}}\Big) v_{\omega}\\
      =& -2m(2m-1) \int_M \frac{h}{f^{2m+1}}g_J\Big(Jdf, J\dot{J} J df\Big) v_{\omega}\\
         & +2(2m-1) \int_M \frac{1}{f^{2m}}{g}_J\Big(Jdf,  J\dot{J} Jdh\Big) v_{\omega}  - \int_M  (\delta J \delta \dot{J}) \Big( \frac{h}{f^{2m-1}}\Big) v_{\omega} \\
       =&\ \ \   2m(2m-1) \int_M \frac{h}{f^{2m+1}}g_J(df, J\dot{J} df) v_{\omega}\\
          & -2(2m-1) \int_M \frac{1}{f^{2m}}{g}_J(df,  J\dot{J} dh) v_{\omega} + \int_M g_{J}\Big(\dot{J},  D J d \Big( \frac{h}{f^{2m-1}}\Big) \Big)v_{\omega} \\
       =&\ \ \  2m(2m-1) \int_M \frac{h}{f^{2m+1}}g_J(df, J\dot{J}  df) v_{\omega}\\
          & -2(2m-1) \int_M \frac{1}{f^{2m}}{g}_J(df,  J\dot{J} dh) v_{\omega} +
                \int_M \frac{1}{f^{2m-1}} g_{J}(\dot{J},  D J d h)v_{\omega}  \\
         &- (2m-1) \int_M \frac{h}{f^{2m}} g_{J}(\dot J, D J df) v_{\omega}
             - 2(2m-1) \int_M \frac{1}{f^{2m}} g_{J}(\dot J, df\otimes Jdh) v_{\omega}\\
         &+2m(2m-1)\int_M \Big(\frac{h}{f^{2m+1}}\Big) g_J(\dot J, df\otimes Jdf) v_{\omega} \\
        =&\ \ \ \int_M \frac{1}{f^{2m-1}} g_{J}(\dot{J},  D J d h)v_{\omega},
\end{aligned}
\end{equation}
where we have used for the last equality the fact that $f$ is a Killing with respect to $g_J$ if and only if the $J$-anti-invariant part of $D J df$ is zero ($D$ stands for the Levi--Civita connection of $g_J$), see e.g. \cite{gauduchon-book}, and that $\dot J$ anti-commutes with $J$.

On the other hand, for any $\varphi \in (C^{\infty}_0(M))^G$, the induced vector field on $\mathcal{C}_{\omega}^G$ is $\hat Z_J= -\mathcal{L}_{Z} J = -2 J (D Jd\varphi^{\sharp})^{\rm sym}$, where $Z= J d\varphi ^{\sharp}$ is the hamiltonian vector field corresponding to $\varphi$, the superscript  `$\sharp$' stands for the map $g_{J}^{-1} : T^*M \to TM$ whereas  the superscript `${\rm sym}$' denotes the symmetric part  with respect to $g_J$.  It follows  that
\begin{equation*}
\begin{split}
\Om^f_{J} (\hat Z _J, \dot{J})& = -\int_M \frac{1}{f^{2m-1}}g_J(D J d\varphi, \dot J) v_{\omega} \\
                                                  & = -\int_{M} \Big(\frac{\dot{s}_{g,J}}{f^{2m+1}}\Big) \varphi  \ v_{\omega},
\end{split}
\end{equation*}
showing that ${s_{J,f}}$ is the $\langle \cdot, \cdot \rangle_{f}$-dual of a momentum map associated to each $\hat Z_J$. The ${\rm Ham}^G(M,\omega)$-equivariance of the map $\mu(J):=-\langle s_{J,f}, \cdot \rangle_f$ is a direct consequence of the fact that the inner product $\langle \cdot, \cdot \rangle_f$ is  ${\rm Ad}$-invariant with respect to ${\rm Ham}^G(M,\omega)$, which in turn uses that $f$ Poisson-commutes with any function in $(C^{\infty}_0(M))^G$. \end{proof}
\begin{rem}\label{AK-futaki}  Instead of $\mathcal{C}_{\omega}^G$, one can consider the larger contractible Frech\'et space $\mathcal{AK}_{\omega}^G$ of $G$-invariant $\omega$-compatible almost-K\"ahler  structures $J$, and extend the definition of $({\bf J}, {\bf \Omega}^f)$ in an obvious way. Then,  the computation in the proof of Theorem~\ref{thm:moment-map-setting} and the formulae in \cite[Ch.~8]{gauduchon-book} (see also \cite[Lemma~2.1 \& Prop.~3.1]{lejmi}) imply that the momentum map for the action of ${\rm Ham}^G(M,\omega)$ on $(\mathcal{AK}_{\omega}^G, {\bf J}, {\bf \Omega}^f)$ is the $\langle \cdot, \cdot \rangle_f$-dual of the functional
\begin{equation*}
s_{J,f}:=2\Big(\frac{2m-1}{m-1}\Big) f^{m+1} \Delta_{J}\Big(\frac{1}{f}\Big)^{m-1} + s_{J} f^2,
\end{equation*}
where $s_{J}$ denotes the {\it Hermitian} scalar curvature of $g_J$ (i.e. the trace of its Ricci form with respect to the canonical Hermitian connection, see e.g \cite{lejmi}) and $\Delta_J$ is the Laplacian of $g_J$.
\end{rem}
As an immediate consequence of Theorem~\ref{thm:moment-map-setting} (or, equivalently, formula \eqref{long-computation} with $h$ a
Killing potential w.r.t. $g_J$) and Remark~\ref{AK-futaki} (using that $\mathcal{AK}_{\omega}^G$ is contractible), we obtain
\begin{cor}\label{c:symplectic-futaki} Let $(M^{2m}, \omega)$ be a compact symplectic manifold or orbifold, $G\subset {\rm Ham}(M,\omega)$ a compact subgroup with Lie algebra $\mathfrak{g}$,  and $f$ an everywhere positive hamiltonian of an element $K \in \mathfrak{g}$. Then
\begin{equation}\label{c0}
c_{\omega, f}: = \Big(\int_M \frac{s_{J,f}}{f^{2m+1}} v_{\omega}\Big) \left/ \Big(\int_M \frac{1}{f^{2m+1}} v_{\omega}\Big)\right.
\end{equation}
and, for each vector field $H \in \mathfrak{g}$ with a hamiltonian $h$, the integral
\begin{equation}\label{F}
\mathfrak{F}^G_{\omega, f} (H) := \int_{M} (s_{J,f}-c_{\omega,f})\frac{h}{f^{2m+1}} v_{\omega}
\end{equation}
are independent of the choice of $J \in \mathcal{C}^G_{\omega}$ and  of the hamiltonian $h$.  In particular, the linear map $\mathfrak{F}^G_{\omega, f} : \mathfrak{g} \to \R$ must be zero,  should a conformally K\"ahler, Einstein--Maxwell metric with conformal factor $1/f^2$ exist for some $J \in \mathcal{C}^G_{\omega}$.
\end{cor}
\begin{defn}\label{d:symplectic-futaki}The linear map $\mathfrak{F}^G_{\omega, f} : \mathfrak{g} \to \R$ defined by \eqref{c0} and \eqref{F} will be called the {\it Futaki invariant} of $(M, \omega, G, f)$.
\end{defn}

\subsection{The Futaki invariant  for a K\"ahler class}\label{s:futaki-kahler-class}  Theorem~\ref{thm:moment-map-setting} implies that finding conformally K\"ahler, Einstein--Maxwell metrics with conformal factor $1/f^2$ corresponds to finding zeroes of a momentum map (in an infinite dimensional setting).  GIT then suggests that a solution to this problem would be unique on each   `complexified' orbit for the  action of ${\rm Ham}^{G}(M,\omega)$,  and would exist in such an orbit under suitable stability conditions. As  observed in \cite{donaldson}, by using the $G$-equivariant Moser lemma, for each $J \in \mathcal{C}^G_{\omega}$, one can identify  the corresponding   `complexified' orbit with the space of $G$-invariant K\"ahler metrics  in the K\"ahler  class $\Omega=[\omega]$ on $(M^{2m}, J)$.  We shall thus adapt our theory to this classical setting.

We thus consider a compact complex K\"ahler  manifold (or orbifold) $(M,J)$, fix a  compact subgroup $G \subset {\rm Aut}_r(M,J)$ in the group of reduced automorphisms of $(M,J)$, and  consider  a fixed K\"ahler class $\Omega$ on $(M,J)$.  Denote by $\mathcal{K}_{\Omega}^G$ the space of $G$-invariant K\"ahler metrics $\omega$ in $\Omega$ (tacitly identifying K\"ahler metrics and forms).  General theory implies that for any $K \in \mathfrak{g}$ and any $\omega \in \mathcal{K}_{\Omega}^G$, the vector field $K$ is  a hamiltonian with respect to $\omega$,  i.e.
\begin{equation*}
\iota_{K}\omega = - d f_{\omega, K}
\end{equation*}
for a smooth function $f_{\omega, K}$ on $M$. Such a function is called a {\it Killing potential} of $K$ with respect to $\omega$. Of course, $f_{K,\omega}$ is defined only up to an additive constant, so in order to avoid this ambiguity, we require
\begin{equation*}
\int_M f_{K,\omega} v_{\omega} = a,
\end{equation*}
where $a$ is a fixed real constant and $v_{\omega}= \frac{1}{m!} \omega^m$ is the volume form of $\omega$. We shall thus denote by $f_{K, \omega, a}$ the unique function satisfying the above two relations.

\smallskip

In what follows, we shall fix the vector field $K$ and,  choosing a reference metric $\omega \in \mathcal{K}_{\Omega}^G$, a constant $a>0$ such that  $f_{K, \omega, a} >0$ on $M$.  It is not hard to prove the following
\begin{lemma}\label{correctness} Suppose that $f_{K,\omega, a}$ is an everywhere positive Killing potential of $K$ with respect to $\omega \in \mathcal{K}_{\Omega}^G$. Then, for any other metric ${\omega}'  \in \mathcal{K}_{\Omega}^G$, the corresponding Killing potential ${f}_{K, \omega', a}$ is everywhere positive.
\end{lemma}
\begin{proof} By the equivariant Moser lemma, for any $\omega, {\omega}' \in \mathcal{K}_{\Omega}^G$ there exist a diffeomorphism $\Phi$,  which commutes with the action of $G$, such that $\Phi^* {\omega}' = \omega$. As $\Phi_* K = K$, it follows that $f_{K,  \omega', a} \circ \Phi = f_{K, \omega, a}$. \end{proof}

Because of this observation, the following problem, which will be of considerable importance to us, is well-defined on $\mathcal{K}_{\Omega}^G$:
\begin{problem}\label{p1} Does there exists an  $\omega \in \mathcal{K}_{\Omega}^G$, such that the conformal Hermitian metric ${\tilde g}_{\omega}(X,Y) = \frac{1}{f_{K, \omega, a}^2} \omega(X, JY)$ is Einstein--Maxwell?
\end{problem}
By the discussion in the previous section, we are looking for metrics $\omega \in \mathcal{K}_{\Omega}^G$ such that the corresponding scalar curvature of $\tilde g$
\begin{equation}
s_{\tilde g, \omega} =  2\Big(\frac{2m-1}{m-1}\Big) f_{K,\omega, a}^{m+1} \Delta_{\omega}\Big(\frac{1}{f_{K,\omega, a}^{m-1}}\Big) + s_{\omega} f^2_{K, \omega, a}
\end{equation}
is constant, where $\Delta_{\omega}$ and $s_{\omega}$ denote the Laplacian and scalar curvature of the K\"ahler metric $g(\cdot, \cdot) = \omega(\cdot, J\cdot)$ corresponding to $\omega$.

\begin{cor}\label{c:complex-futaki} Let $(M,J)$ be a compact complex manifold of complex dimension $m \ge 2$, $G \subset {\rm Aut}_{r}(M,J)$ a compact subgroup, $K \in \mathfrak{g}={\rm Lie}(G)$ a vector field, and $a>0$ a real constant such that $K$ has a positive Killing potential $f_{K,\omega, a}$ with integral equal to $a$ with respect to any $\omega \in \mathcal{K}^G_{\Omega}$. Then,  for any vector field $H \in {\rm Lie}(G)$ with Killing potential $h_{H, \omega, b}$ with respect to $\omega$, the integral
\begin{equation*}
\int_M s_{\tilde g, \omega} h_{H,\omega, b}\Big(\frac{1}{f_{K,\omega, a}}\Big)^{2m+1} v_{\omega}
\end{equation*}
is independent of the choice of $\omega$ and $b$. In particular,
\begin{equation}\label{c-complex}
c_{\Omega,K,a}:=\frac{\int_M s_{\tilde g, \omega} \Big(\frac{1}{f_{K,\omega, a}}\Big)^{2m+1}v_{\omega}}{\int_M \Big(\frac{1}{f_{K,\omega, a}}\Big)^{2m+1}v_{\omega}}
\end{equation}
is a constant independent of $\omega \in \mathcal{K}_{\Omega}^G$, and  for any vector field $H \in \mathfrak{g}$
\begin{equation}\label{futaki-integral}
\mathfrak{F}^G_{\Omega, K,a} (H):= \int_M \Big( \frac{s_{\tilde g, \omega}-c_{\Omega, K,a}}{f_{K,\omega, a}^{2m+1}}\Big)  h_{H, \omega, b}\ v_{\omega}
\end{equation}
introduces a linear function on $\mathfrak{g}$ which must vanish, should a solution to Problem~\ref{p1} exist.
\end{cor}
\begin{proof} As in the proof of Lemma~\ref{correctness}, we can use $G$-equivariant Moser lemma  in order to reduce the problem to one concerning $\omega$-compatible $G$-invariant complex structures  in the same connected component of $\mathcal{C}^{G}_{\omega}$.   The claim then follows from Corollary~\ref{c:symplectic-futaki}, or equivalently, by \eqref{long-computation}, by taking $h$ to be  another Killing potential (and thus satisfying $g_J (D J h, \dot J)=0$ as $\dot J$ and $J$ anti-commute. \end{proof}
\begin{defn}\label{def:complex-futaki} The linear  map $\mathfrak{F}^G_{\Omega, K,a} : \mathfrak{g} \to \R$ will be referred to as the Futaki invariant associated to $(\Omega, G, K, a)$.
\end{defn}

\section{The toric case}\label{s:toric} We now specialize to the case when $G=\T$ is an $m$-dimensional torus, where we recall,  $m$ is the complex dimension of $(M,J)$. Before summarizing some of the standard theory of K\"ahler toric metrics, we remark
that we will call a conformally K\"ahler, Einstein-Maxwell metric $\tilde g$ {\em toric} if $Isom(\tilde g) \cap Aut_r(M,J)$ contains a torus $\T$ of dimension $m$.

We shall use the formalism of \cite{guillemin,abreu}, i.e. we fix a $\T$-invariant  symplectic form $\omega$ in $\Omega$,   a positive  hamiltonian function $f$ for $K$, and vary the K\"ahler metrics within the set $\mathcal{C}_{\omega}^{\T}$ of $\T$-invariant, $\omega$-compatible complex structures on $(M, \omega, \T)$.  It is a basic fact of the theory that any two elements $J', J \in \mathcal{C}_{\omega}^{\T}$ are biholomorphic under a $\T$-equivariant diffeomorphism $\Phi$ which acts trivially on the cohomology class of  $[\omega]$, i.e. $\Phi^*\omega$  is a K\"ahler form in  $\mathcal{K}_{[\omega]}^{\T}$ on $(M,J)$. The converse is also true by the equivariant Moser lemma.

As the  action of $\T\subset {\rm Aut}_r(M,J)$ is individually hamiltonian (i.e. each $K \in {\rm Lie}(\T)$ admits a hamiltonian function with respect to $\omega$) it is therefore hamiltonian, so one can use the Delzant description~\cite{Delzant,LT} of $(M, \omega, \T)$  in terms of  the  corresponding momentum image  $\Delta = \mu(M) \subset \mathfrak{t}^*$ and a set of non-negative defining affine functions ${\bf L}=(L_1, \ldots, L_d)$ for $\Delta$, defined on the vector space $\mathfrak{t}^* = {\rm Lie}(\T) \cong \R^m$. In fact, the theory extends naturally to orbifolds, in which case the data $(\Delta, {\bf L})$ is refereed to as a {\it labelled  rational Delzant polytope}, see \cite{LT}, and  classifies $(M, \omega, \T)$ up to an equivariant symplectomorphism.

In the following, we let $\Lambda$ denote the lattice of circle subgroups of $\T$, defined up to a factor $2\pi$ as the kernel of the exponential map from $\tor$ to $\T$, i.e. $\T = \tor / 2\pi \Lambda$.
On the union $M^0:=\mu^{-1}(\Delta^0)$ of the generic orbits of the $\T$-action, the K\"ahler metrics in $\mathcal{C}^{\T}_{\omega}$
have  a general expression due to
V.~Guillemin~\cite{guillemin}. In this description, the
momentum map $\mu\colon M^0\to \mathfrak{t}^*$ is supplemented by angular coordinates
$\ang\colon M\to \tor/2\pi\Lambda$ such that the kernel of $\d\ang$ is
orthogonal to the torus orbits. These action-angle coordinates $(\mu,\ang)$
identify each tangent space to $M^0$ with $\tor\oplus \tor^*$, and the
symplectic form is $\omega=\ip{\d\mu\wedge \d\ang}$, where $\ip{\cdot}$
denotes contraction of $\tor$ and $\tor^*$.  Hence invariant
$\omega$-compatible K\"ahler metrics on $M^0$, i.e. toric metrics, have the form
\begin{equation}\label{toricmetric}
g_J=\ip{\d\mu, {\mathbf G} , \d\mu}+ \ip{ \d\ang,{\mathbf H}, \d\ang},
\end{equation}
where ${\mathbf G}$ is a positive definite $S^2\tor$-valued function of $\mu$
(with $S^2\tor$ the symmetric tensor product of $\tor$),
${\mathbf H}$ is its point-wise inverse in $S^2\tor^*$ (at each point,
${\mathbf G}$ and ${\mathbf H}$ define mutually inverse linear maps $\tor^*
\to\tor$ and $\tor\to\tor^*$) and $\ip{\cdot,\cdot,\cdot}$ denotes the
point-wise contraction $\tor^* \times S^2\tor \times \tor^* \to \R$ or the
dual contraction.  The corresponding almost complex structure is defined by
\begin{equation}\label{toricJ}
J \d\ang = -\ip{ {\mathbf G}, \d\mu},
\end{equation}
and $J$ is integrable if and only if ${\mathbf G}$ is the Hessian of a
 smooth convex function $u$, called the symplectic potential of $(\omega, J)$~\cite{guillemin,abreu}.

Necessary and sufficient conditions for $\mathbf H$ to come from a globally
defined metric on $M$ are obtained in \cite{abreu,ACGT2,Donaldson2}.  Here we
use the first order boundary conditions given in~\cite[\S1]{ACGT2}.  In order
to state them, we denote by ${\tor}_\Fa \sub {\tor}$ (for any face $\Fa\sub
\Delta$) the vector subspace spanned by the inward normals $u_j \in {\tor}$ to
facets containing $\Fa$. Thus the tangent plane to points in the interior $\Fa^0$ of $\Fa$
is the annihilator ${\tor}^0_\Fa\cong({\tor}/{\tor}_\Fa)^*$ of ${\tor}_\Fa$ in
$\tor^*$.

\begin{prop}\label{p:toric-ccs} Let $(M,\omega)$ be a compact toric symplectic
$2m$-manifold or orbifold with momentum map $\mu \colon M\to \Delta
\subset \ \tor^*$, and ${\mathbf H}$ be a positive definite
$S^2\tor^*$-valued function on $\Delta^0$.  Then ${\mathbf H}$ defines a
$\T$-invariant, $\omega$-compatible almost K\"ahler metric $g$ on $M$ via
\eqref{toricmetric} if and only if it satisfies the following
conditions\textup:
\begin{bulletlist}
\item \textup{[smoothness]} ${\mathbf H}$ is the restriction to $\Delta^0$ of a
smooth $S^2\tor^*$-valued function on $\Delta$\textup;
\item \textup{[boundary values]} for any point $\xi$ on the facet $\Fa_j
\subset \Delta$ with inward normal $u_j=dL_j \in \tor$,
\begin{equation}\label{toricboundary}
{\mathbf H}_{\xi}(u_j, \cdot) =0\qquad {and}\qquad (d{\mathbf
H})_{\xi}(u_j,u_j) = 2 u_j,
\end{equation}
where the differential $d{\mathbf H}$ is viewed as a smooth $S^2\tor^*\otimes
{\tor}$-valued function on $\Delta$\textup;
\item \textup{[positivity]} for any point $\xi$ in the interior of a face $\Fa
\sub \Delta$, ${\mathbf H}_{\xi}(\cdot, \cdot)$ is positive definite when
viewed as a smooth function with values in $S^2({\tor}/{\tor}_\Fa)^*$.
\end{bulletlist}
\end{prop}
It follows that, up to a $\T$-equivariant isometries, the space $\mathcal{C}_{\omega}^{\T}$ can be identified with the function space $\mathcal{S}(\Delta, \bf L)$ of all smooth, strictly convex functions $u$ on $\Delta^0$ such that   ${\bf H}^u := {\rm Hess}(u)^{-1}$  satisfies the conditions of Proposition~\ref{p:toric-ccs}. A key feature of this description is the remarkably simple expression, found by Abreu in \cite{abreu0},  for the scalar curvature $s_J$ of the K\"ahler metric determined by $J \in \mathcal{C}^{\T}_{\omega}$, in terms of the corresponding symplectic potential $u \in \mathcal{S}(\Delta, \bf L)$. We introduce an identification $\tor \cong \R^m$ which gives rise to momentum/angle coordinates $(\mu_1, \ldots, \mu_m, t_1, \ldots, t_m)$ around each point of $M^0$. The scalar curvature of $g_J$ is then given by
\begin{equation}\label{s-abreu}
s_J = -\sum_{i,j=1}^{m}  {\bf H}^u_{ij, ij}= - \sum_{i,j=1}^m u^{,ij}_{,ij},
\end{equation}
where,  for a smooth function $\varphi$ of $\mu=(\mu_1, \ldots, \mu_m)$,  we denote by $\varphi_{, i} $ the partial derivative $\partial \varphi/\partial \mu_i$ and  let ${\bf H}^u:= (u_{,ij})^{-1}$ be the inverse of the Hessian.

\smallskip
Note that \eqref{toricmetric} and the symplectic form are now given by
\begin{equation}\label{toric-metric-u}
g_{J}=\sum_{i,j=1}^m \Big(u_{, ij} d\mu_i d\mu_j + {\bf H}^{u}_{ij} dt_i dt_j\Big),  \ \ \omega = \sum_{i=1}^m d\mu_i \wedge dt_i,
\end{equation}
whereas \eqref{toricJ} reads as
\begin{equation*}
J d \mu_i = \sum_{j=1}^m {\bf H}^{u}_{ij} dt_j,
\end{equation*}
so that, for any smooth function $\varphi$ of $\mu_1, \ldots, \mu_m$,  one has
\begin{equation}\label{laplace}
\begin{split}
dd_J^c \varphi &= \sum_{i,j,k=1}^{m}\Big(\varphi_{,j}{\bf H}^{u}_{jk}\Big)_{,i} d\mu_i \wedge dt_k, \\
\Delta_J \varphi & = - \sum_{i,j=1}^m \Big( \varphi_{,i} {\bf H}^{u}_{ij}\Big)_{,j} \\
                          &= -\sum_{i,j=1}^m \varphi_{,ij} {\bf H}^{u}_{ij}  - \sum_{i,j=1}^m \varphi_{,i} {\bf H}^{u}_{ij, j}.
\end{split}
\end{equation}

\smallskip

In what follows,  $f$ will be a fixed affine function which is positive on $\Delta$ (so that the pull-back of $f$ by $\mu$ is a Killing potential for any K\"ahler metric $g_J$ corresponding to an element $J\in \mathcal{C}_{\omega}^{\T}$). In particular, $f_{,ij}=0$. Thus, one computes for any smooth matrix-valued function ${\bf H}= ({\bf H}_{ij})$ on $\Delta$
\begin{equation}\label{computation}
\begin{split}
\sum_{i,j=1}^m \Big(\frac{{\bf H}_{ij}}{f^{2m-1}}\Big)_{,ij} =  &  \frac{1}{f^{2m-1}}  \sum_{i,j=1}^m {\bf H}_{ij,ij}\\
                                                                                      & + \frac{2m(2m-1)}{f^{2m+1}} \sum_{i,j=1}^m {\bf H}_{ij}f_{,i} f_{,j} \\
                                                                                      & - \frac{2(2m-1)}{{f}^{2m}} \sum_{i,j=1}^m {\bf H}_{ij,j} f_{,i}
\end{split}
\end{equation}
From \eqref{s-symplectic}, \eqref{s-abreu} and \eqref{laplace}, and the above equality one gets for the scalar curvature $s_{J,f}$ of $\frac{1}{f^2}g_J$
\begin{equation}\label{sg}
\frac{s_{J, f}}{f^{2m+1}} = - \sum_{i,j=1}^m \Big(\frac{1}{{f}^{2m-1}} {\bf H}^u_{ij}\Big)_{,ij}.
\end{equation}
\begin{rem} Formula \eqref{sg} can alternatively be obtained by Theorem~\ref{thm:moment-map-setting} and a direct check as in \cite{Do-02} that $- \sum_{i,j=1}^m \Big(\frac{1}{{f}^{2m-1}} {\bf H}^u_{ij}\Big)_{,ij}$ computes the momentum map of the action of ${\rm Ham}^{\T}(M,\omega)$ on $\mathcal{S}(\Delta, \bf L)$.
\end{rem}
Thus, Problem~\ref{p1} reduces to finding $u \in \mathcal{S}(\Delta, {\bf L})$ such that
\begin{equation}\label{modified-abreu}
- \sum_{i,j=1}^m \Big(\frac{1}{{f}^{2m-1}} {\bf H}^u_{ij}\Big)_{,ij} = \frac{c_{\Delta, {\bf L}, f}}{f^{2m+1}},
\end{equation}
where $c_{\Delta, {\bf L},f}$ is (necessarily) a constant determined by $(\Delta, {\bf L})$ (see \eqref{c} for a precise definition). Note that this is very similar,  and yet different (when $f$ is not constant) from the general problem of prescribing the value of $- \sum_{i,j=1}^m {\bf H}^u_{ij,ij}$ discussed in \cite{Do-02};  multiplying \eqref{modified-abreu} by the positive factor $f^{2m-1}$, we obtain an equation for $u$ which is of the type studied in \cite[Sect.~3.1]{ACGT4}. In the reminder of this section,  we briefly review the theory pioneered in \cite{Do-02}, which has been adapted to this class of equations in \cite{ACGT4}.

\smallskip
The standard Lebesgue measure $d\mu = d\mu_1\wedge \ldots \wedge d\mu_m$ on $\tor^* \cong \R^m$ and the affine labels ${\bf L}=(L_1, \ldots, L_d)$ of $\Delta$ induce a measure $d\sigma$ on each facet $F_i \subset \partial \Delta$ by letting
\begin{equation}\label{boundary-measure}
dL_i \wedge d\sigma = - d\mu.
\end{equation}
With this normalization, we have the following
\begin{lemma}\label{by-parts}  Let ${\bf H}$ be any smooth $S^2\tor^*$-valued function on $\Delta$ which satisfies the boundary condition \eqref{toricboundary} of Proposition~\ref{p:toric-ccs}, but not necessarily the positivity condition. Then, for any smooth functions $\varphi, \psi$ on $\tor^*$
\begin{equation*}
\int_{\Delta} \Big(\sum_{i,j=1}^m \Big(\psi {\bf H}_{ij}\Big)_{,ij}\Big)  \varphi \ d\mu =\int_{\Delta} \Big(\sum_{i,j=1}^m (\psi {\bf H}_{ij}) \varphi_{,ij} \Big)d\mu\\
-2\int_{\partial \Delta} \varphi \psi d\sigma .\end{equation*}
\end{lemma}
\begin{proof} The proof is elementary and uses integration by parts: recall that for any smooth $\tor^*$-valued function $V=(V_1, \ldots, V_m)$ on $\tor^*$, Stokes theorem gives
\begin{equation}\label{integration-by-parts}
\int_{\Delta} \sum_{j=1}^m V_{j,j} d\mu = -\sum_{k=1}^d \int_{F_k} \langle V, dL_k \rangle d\sigma,
\end{equation}
where we have used the convention \eqref{boundary-measure} for $d\sigma$.
Let $\tilde{\bf H}= \psi {\bf H}$. We shall use the identity
\begin{equation}\label{basic-formula}
\sum_{i,j=1}^m \varphi_{,ij} \tilde{\bf H}_{ij} = \sum_{i,j=1}^m \varphi \tilde{\bf H}_{ij,ij} - \sum_{j=1}^m V_{j,j} ,
\end{equation}
where
\begin{equation}
V_j:=\varphi \sum_{i=1}^m \tilde{\bf H}_{ij,i} - \sum_{i=1}^m \varphi_{,i}\tilde{\bf H}_{ij}.
\end{equation}
It follows by \eqref{basic-formula} and \eqref{integration-by-parts} that
\begin{equation}\label{stokes}
\int_{\Delta} \sum_{i,j=1}^m \varphi_{,ij} \tilde{\bf H}_{ij} d\mu = \int_{\Delta} \sum_{i,j=1}^m \varphi \tilde{\bf H}_{ij,ij} + \sum_{k=1}^{d}\int_{F_k}\langle V, dL_k\rangle d\sigma.
\end{equation}
On each facet $F_k$ we have, using a basis $\{e_1,\ldots e_m\}$ of $\tor$,
\begin{equation}\label{detailed}
\begin{split}
\langle dL_k, V\rangle  &= \sum_{j=1}^m \langle dL_k, e_j^*\rangle V_j \\
                                          &= \sum_{j=1}^m \langle dL_k, e_j^*\rangle (\varphi \sum_{i=1}^m \tilde{\bf H}_{ij,i} - \sum_{i=1}^m \varphi_{,i}\tilde{\bf H}_{ij})\\
                                          &= \varphi \sum_{i=1}^m \langle d \tilde{\bf H}(dL_k, e_i), e_i^*\rangle - \sum_{i=1}^m \tilde{\bf H}(dL_k, d\varphi).
\end{split}
\end{equation}
Using \eqref{toricboundary}, we have  $\tilde{\bf H}(dL_k, \cdot) =0$ on $F_k$ and $d \tilde{\bf H}(dL_k, e_i) = \psi d{\bf H}(dL_k, e_i)$. It thus follows (e.g. by choosing a basis $e_1 = dL_k, e_2, \ldots, e_m$ with $e_j^*$ tangent to $F_k$ for $j>1$) that
\begin{equation*}
\sum_{i=1}^m  \langle d \tilde{\bf H}(dL_k, e_i), e_i^*\rangle = \psi \sum_{i=1}^m  \langle d {\bf H}(dL_k, e_i), e_i^*\rangle =2\psi .
\end{equation*}
Substituting back in \eqref{detailed} and then the result in \eqref{stokes} completes the proof. \end{proof}
Applying the above lemma for $\psi = \frac{1}{f^{2m-1}}$ and ${\bf H}$ smooth on $\Delta$  and satisfying \eqref{toricboundary}, one has
\begin{equation}\label{special-case}
\begin{split}
-\int_{\Delta} \Big(\sum_{i,j=1}^m \Big(\frac{1}{{f}^{2m-1}} {\bf H}_{ij}\Big)_{,ij}\Big)  \varphi \ d\mu = &-\int_{\Delta} \frac{1}{f^{2m-1}}\Big(\sum_{i,j=1}^m {\bf H}_{ij} \varphi_{,ij} \Big)d\mu\\
& +2\int_{\partial \Delta} \Big(\frac{\varphi}{f^{2m-1}} \Big) d\sigma
\end{split}
\end{equation}
for any smooth function $\varphi$; specializing to $\varphi$ affine, one sees that the rhs is manifestly independent of ${\bf H}$, in accordance with Corollary~\ref{c:symplectic-futaki}.

In order to relate this to the Futaki invariant, consider the following quantity, for any constant $c$, any smooth ${\bf H}$ which satisfies the boundary conditions of Proposition~\ref{p:toric-ccs} (but is not necessarily positive definite or an inverse of a Hessian) and any smooth function $\varphi$:
\begin{equation}\label{Donaldson-Futaki-0}
\begin{split}
\mathfrak{F}_{\Delta, {\bf L}, f} (\varphi) : = & 2\int_{\partial \Delta} \Big(\frac{\varphi}{f^{2m-1}} \Big) d\sigma -c\,\int_{\Delta} \frac{\varphi}{f^{2m+1}}d\mu \\
                                                                 =&-\int_{\Delta} \Big(\frac{c}{f^{2m+1}} + \sum_{i,j=1}^m \Big(\frac{1}{{f}^{2m-1}} {\bf H}_{ij}\Big)_{,ij}\Big)  \varphi \ d\mu\\
                                                                    &  +\int_{\Delta} \frac{1}{f^{2m-1}}\Big(\sum_{i,j=1}^m {\bf H}_{ij} \varphi_{,ij} \Big)d\mu,
\end{split}
\end{equation}
where the second and third line follow from the first by \eqref{special-case}.
Now, in the special case where for some $u \in \mathcal{S}(\Delta, {\bf L})$, ${\bf H}={\bf H}^u$ {\em is} a positive
definite inverse of a Hessian while $c$ is the constant $c_{\omega,f}$ of Corollary \ref{c:symplectic-futaki}, it follows that $\mathfrak{F}_{\Delta, {\bf L}, f}$ equals $\frac{1}{(2\pi)^m}$ times the Futaki invariant, when restricted to functions $\varphi$ which are affine in momenta. This can be seen from the definition \eqref{F} together with \eqref{sg} and \eqref{special-case}.

Defining
\begin{equation}\label{c}
c_{\Delta, {\bf L}, f} : = 2\left(\frac{\int_{\partial \Delta} \frac{1}{f^{2m-1} }d\sigma}{\int_{\Delta} \frac{1}{f^{2m+1} }d\mu}\right),
\end{equation}
we now obtain the following  straightforward analogue of the results in \cite{Do-02}:
\begin{thm}\label{toric-futaki} If there exists a smooth $(m\times m)$-matrix valued function ${\bf H}$ defined on $\Delta$, which satisfies the boundary conditions \eqref{toricboundary} and the equation
\begin{equation}\label{matrix-abreu}
- \sum_{i,j=1}^m \Big(\frac{1}{{f}^{2m-1}} {\bf H}_{ij}\Big)_{,ij} = \frac{c}{f^{2m+1}},
\end{equation}
for a constant $c$, then $c$ is the positive constant $c_{\Delta, {\bf L}, f}$  given by \eqref{c}, and the Donaldson--Futaki invariant  defined by
\begin{equation}\label{Donaldson-Futaki-1}
\mathfrak{F}_{\Delta, {\bf L}, f} (\varphi) =  2\int_{\partial \Delta} \Big(\frac{\varphi}{f^{2m-1}} \Big) d\sigma -(c_{\Delta, {\bf L}, f}) \int_{\Delta} \frac{\varphi}{f^{2m+1}}d\mu
\end{equation} vanishes for any affine function $\varphi$. If, furthermore, ${\bf H}$ is positive definite on $\Delta^0$, in particular if ${\bf H}= {\bf H}^{u}$ for a solution $u\in \mathcal{S}(\Delta, {\bf L})$ of \eqref{modified-abreu}, then
\begin{equation}\label{K-stability}
\mathfrak{F}_{\Delta, {\bf L}, f} (\varphi)\ge 0
\end{equation}  for any smooth convex function $\varphi$ on $\Delta$,  with equality in \eqref{K-stability} if and only if $\varphi$ is affine.
\end{thm}
\begin{proof} By the second line of the general formula \eqref{Donaldson-Futaki-0} along with the first line, computed for $\varphi=1$, for any ${\bf H}$ which satisfies \eqref{matrix-abreu} and the boundary conditions \eqref{toricboundary} of Proposition~\ref{p:toric-ccs}, $c$ must be given by \eqref{c}. Furthermore, for such ${\bf H}$ and any smooth function $\varphi$, the Donaldson--Futaki invariant \eqref{Donaldson-Futaki-1} is given by
\begin{equation}\label{Futaki-Hessian}
\mathfrak{F}_{\Delta, {\bf L}, f} (\varphi) =\int_{\Delta} \frac{1}{f^{2m-1}}\Big(\sum_{i,j=1}^m {\bf H}_{ij} \varphi_{,ij} \Big)d\mu.
\end{equation}
Specializing the above identity to the case where ${\bf H}$ is positive definite and $\varphi$ is affine (resp. convex) yields the rest of the conclusions.\end{proof}
\begin{rem}
It follows from formula \eqref{c} that any compact, toric  conformally-K\"ahler, Einstein--Maxwell manifold or orbifold
must have {\it positive} scalar curvature.
\end{rem}
Formula \eqref{Donaldson-Futaki-1} makes sense not only for smooth functions $\varphi$, but  more generally  for continuous functions on $\Delta$. Following \cite{Do-02}, we introduce
\begin{defn}\label{stability} Let $(\Delta, {\bf L})$ be a labelled Delzant polytope and $f$ a positive affine function on $\Delta$. We say that $(\Delta, {\bf L}, f)$ is $K$-semistable if $\mathfrak{F}_{\Delta, {\bf L}, f} (\varphi)\ge 0$ for any piecewise affine linear (PL)  convex function $\varphi$ on $\Delta$. It is $K$-polystable if, furthermore, the equality $\mathfrak{F}_{\Delta, {\bf L}, f} (\varphi)= 0$ is only possible for $\varphi$ affine linear. \end{defn}
\noindent According to \cite[Conjecture 7.2.2]{Do-02}, the main conjecture in this theory is
\begin{conj}\label{c:donaldson} There exists a solution of \eqref{modified-abreu} in $\mathcal{S}(\Delta, {\bf L})$ if and only if $(\Delta, {\bf L}, f)$ is $K$-polystable.
\end{conj}
\noindent As the smooth convex functions are dense in the space of  convex functions over $\Delta$, Theorem~\ref{toric-futaki}  already  shows that the existence of a solution of  \eqref{modified-abreu} implies $K$-semistability of $(\Delta, {\bf L}, f)$. The argument from \cite{ZZ} show that more is true.
\begin{cor}\label{c:K-stability} If \eqref{modified-abreu}  admits a solution $u \in \mathcal{S}(\Delta, {\bf L})$ then $(\Delta, {\bf L}, f)$ is K-polystable.
\end{cor}
\begin{proof} The proof is elementary and uses integration by parts (similar to the proof of Lemma~\ref{by-parts}) in order to obtain a distribution analogue of \eqref{Futaki-Hessian} in the case when $\varphi$ is a PL convex function.  For simplicity, and as we shall later use this case, we  only check that if \eqref{modified-abreu} admits a solution $u \in \mathcal{S}(\Delta, {\bf L})$ then $\mathfrak{F}_{\Delta, {\bf L}, f}$ is strictly positive for a {\it simple crease} convex PL function $\varphi$, i.e. $\varphi = {\rm max}(L, \tilde{L})$ where $L$ and $\tilde{L}$  are affine linear functions on $\tor^*$  with $L- \tilde{L}$ vanishing in the interior of $\Delta$; as $\mathfrak{F}_{\Delta, {\bf L}, f}(\tilde{L})=0$ by Theorem~\ref{toric-futaki}, we can assume without loss of generality that $\tilde L \equiv 0$.

Denote by $F= \Delta \cap \{L=0\}$ the `crease'  of $\varphi$. This introduces a partition $\Delta= \Delta' \cup \Delta''$ of $\Delta$ into $2$ sub-polytopes  with $F$ being a common facet of the two. Notice that $\varphi$ is affine linear over each $\Delta'$ and $\Delta''$, being zero over $\Delta''$, say. Furthermore, $dL$ defines an inward normal for $\Delta'$ (by convexity) and we use \eqref{boundary-measure} to define a measure $d\sigma$ on $\partial \Delta'$.  Let us write $\partial \Delta' = F \cup \partial \Delta_0$, where $\partial \Delta_0$ is the union of facets of $\Delta'$ which belong to $\partial \Delta$. We then have (using that $\varphi \equiv 0$ on $\Delta''$)
\begin{equation*}
\begin{split}
\mathfrak{F}_{\Delta, {\bf L}, f} (\varphi) = &\, 2\int_{\partial \Delta} \Big(\frac{\varphi}{f^{2m-1}} \Big) d\sigma -(c_{\Delta, {\bf L}, f}) \int_{\Delta} \frac{\varphi}{f^{2m+1}}d\mu \\
= &\, 2\int_{\partial \Delta_0} \Big(\frac{\varphi}{f^{2m-1}} \Big) d\sigma - (c_{\Delta, {\bf L}, f}) \int_{\Delta'} \frac{\varphi}{f^{2m+1}}d\mu.
  \end{split}
 \end{equation*}
We now use \eqref{stokes}  and \eqref{detailed} over $\Delta'$ with $\tilde{\bf H}_{ij} = \frac{1}{f^{2m-1}} {\bf H}^{u}_{ij}$, where $u\in \mathcal{S}(\Delta, {\bf L})$ is a solution of \eqref{modified-abreu}, i.e,  $-\sum_{ij} \tilde{\bf H}_{ij,ij}= \frac{c_{\Delta, {\bf L}, f}}{f^{2m+1}}$.  Noting that $\varphi$ is affine over $\Delta'$, i.e. $\varphi_{,ij}=0$ we have
\begin{equation*}
\begin{split}
- (c_{\Delta, {\bf L}, f})\int_{\Delta'} \frac{\varphi }{f^{2m+1}} d\mu = &\int_{\Delta'} \varphi \Big(\sum_{i,j=1}^m \tilde{\bf H}_{ij,ij}\Big) d\mu \\
                                                                                                                     = &-2\int_{\partial \Delta_0}\frac{\varphi}{f^{2m-1}}  d\sigma\\
                                                                                                                       &- \int_{F}\Big(\varphi \sum_{i=1}^m \langle d \tilde{\bf H}(dL, e_i), e_i^*\rangle\Big) d\sigma \\
                                                                                                                    & + \int_{F} \tilde{\bf H}(dL, dL) d\sigma.
\end{split}
\end{equation*}
As $\varphi$ vanishes on $F$, the term at the third line is zero, so that
\begin{equation}\label{Futaki-segment}
\mathfrak{F}_{\Delta, {\bf L}, f} (\varphi) =  \int_{F} \frac{1}{f^{2m-1}} {\bf H}(dL, dL) d\sigma >0,
\end{equation}
as ${\bf H}= {\bf H}^{u}$ is positive definite over $\Delta^0$. \end{proof}

\smallskip
We end this section by showing the uniqueness of the solution of \eqref{modified-abreu}, adapting an argument originally due to D. Guan~\cite{Guan} in the extremal toric case. Consider the functional
\begin{equation}\label{K-energy}
\mathfrak{E}_{\Delta, {\bf L}, f}(u) = \mathfrak{F}_{\Delta, {\bf L}, f}(u)  - \int_{\Delta} \frac{\log \det \Hess (u)}{f^{2m-1}} d\mu,
\end{equation}
defined for elements  $u \in \mathcal{S}(\Delta, {\bf L})$ and which we shall refer to as the {\it $K$-energy} of $(\Delta, {\bf L}, f)$. Donaldson shows~\cite{Do-02} that the convexity of $u$ implies that the second term on the rhs of the above formula is a well-defined integral taking values in $(-\infty, \infty]$. But one can alternatively introduce a reference element $u_c \in \mathcal{S}(\Delta, {\bf L})$ and consider  instead the functional
\begin{equation*}
u \mapsto \mathfrak{F}_{\Delta, {\bf L}, f}(u)  - \int_{\Delta} \frac{(\log \det \Hess (u)- \log \det \Hess (u_c))}{f^{2m-1}} d\mu
\end{equation*}
which is finite by \cite[Thm.~2]{abreu}. Using the formula $d \log \det {\bf A} = \trace {\bf A}^{-1} d {\bf A}$ for any non-degenerate matrix ${\bf A}$ and \eqref{Donaldson-Futaki-0}, one computes the first variation of  $\mathfrak{E}_{\Delta, {\bf L}, f}$ at $u$ in the direction of $\dot u$
\begin{equation*}
\begin{split}
\Big(d \mathfrak{E}_{\Delta, {\bf L}, f}\Big)_{u} (\dot{u})=&  \mathfrak{F}_{\Delta, {\bf L}, f}(\dot{u}) -\int_{\Delta} \frac{1}{f^{2m-1}}\sum_{i,j=1}^m {\bf H}^{u}_{ij} \dot{u}_{,ij} d\mu \\
                                                                               =& \int_{\Delta}\Big[\Big( -\sum_{i,j=1}^m \Big(\frac{1}{{f}^{2m-1}} {\bf H}_{ij}\Big)_{,ij}\Big) - \frac{c_{\Delta,{\bf L}, f}}{f^{2m+1}}\Big]  \dot{u} \ d\mu,
\end{split}
\end{equation*}
showing that the critical points  of $\mathfrak{E}_{\Delta, {\bf L}, f}$ are precisely the solutions of \eqref{modified-abreu}. Furthermore, using $d{\bf A}^{-1} = -{\bf A}^{-1} d{\bf A} {\bf A}^{-1}$, the second variation of $\mathfrak{E}_{\Delta, {\bf L}, f}$ at $u$ in the directions of $\dot{u}$ and $\dot{v}$ is
\begin{equation*}
\Big(d^2 \mathfrak{E}_{\Delta, {\bf L}, f}\Big)_{u}(\dot{u}, \dot{v}) =\int_{\Delta} \frac{1}{f^{2m-1}} \trace \Big(\Hess(u) \Hess(\dot{u}) \Hess(u) \Hess(\dot{v})\Big) d\mu,
\end{equation*}
showing $\mathfrak{E}_{\Delta, {\bf L}, f}$ is convex. In fact, as $\Hess(u)$ is positive definite and $\Hess(\dot{u})$ is symmetric, the vanishing  $\Big(d^2 \mathfrak{E}_{\Delta, {\bf L}, f}\Big)_{u}(\dot{u}, \dot{u}) = 0$ is equivalent to $\dot{u}$
being affine.

One can show (e.g. by using \cite[Thm.~2]{abreu}) that for any two elements $u_1, u_2 \in \mathcal{S}(\Delta, {\bf L})$, $u(t) = tu_1 + (1-t)u_2, \ t \in [0,1]$ is curve in $\mathcal{S}(\Delta, {\bf L})$ with tangent vector $\dot{u}= u_1-u_2$. It follows by the convexity of $\mathfrak{E}_{\Delta, {\bf L}, f}$ that if $u_1$ and $u_2$ are two solutions of \eqref{modified-abreu} (equivalently, critical points of $\mathfrak{E}_{\Delta, {\bf L}, f}$), then $u_1-u_2$ must be affine. We thus have proved the following
\begin{thm}\label{uniqueness} Any two solutions $u_1, u_2\in \mathcal{S}(\Delta, {\bf L})$ of \eqref{modified-abreu} differ by an affine function. In particular, on a compact toric K\"ahler manifold or orbifold $(M, \omega, \T)$, for any  fixed positive affine function  in momenta $f$, there exists at most one, up to a $\T$-equivariant isometry, $\omega$-compatible $\T$-invariant K\"ahler metric  $g_J \in \mathcal{C}_{\omega}^{\T}$ for which $g_{J,f} = \frac{1}{f^2} g_J$ is a conformally K\"ahler, Einstein--Maxwell metric.
\end{thm}

\section{Einstein--Maxwell toric 4-orbifolds with $b_2 \le 2$}\label{s:examples}

 \subsection{Weighted projective planes}\label{s:wpp} Recall that (see e.g. \cite{abreu,ACGT2}) any compact toric symplectic $2m$-orbifold whose momentum image $\Delta$ is a simplex in $\R^m$ is (orbifold) covered by a weighted projective space $\C P^{m}_{a_0, \ldots, a_m}$, where $a_0\ge a_1 \ge \cdots \ge  a_m>0$ are integers  with ${\rm g.c.d.}(a_0, \ldots, a_m)=1$.  Note that, as the second Betti number  $b_2$ of such an orbifold equals $1$, the symplectic form is unique up to homothety and symplectomorphism.  Furthermore, as pointed out by R. Bryant~\cite{Bryant}, each weighted projective space admits an  extremal K\"ahler metric $g_{BF}$, which is also Bochner--flat~\cite{Bryant,Webster}. In complex dimension $m=2$, the Bochner tensor of a K\"ahler orbi-surface coincides with its anti-self-dual Weyl tensor. It thus follows that the extremal K\"ahler  metric of $\C P^2_{a_0,a_1,a_3}$ is self-dual, in particular has vanishing Bach tensor~\cite{De}.  A. Derdzinski~\cite{De}  showed that
the metric $\tilde g=\frac{1}{s_K^2} g_{BF}$ is then Einstein (and therefore conformally K\"ahler,  Einstein--Maxwell) whenever the scalar curvature $s_{BF}\neq 0$. It is shown in \cite{Bryant,DG} that this latter condition holds everywhere  on  $\C P^{m}_{a_0, a_1, a_2}$ if and only if $a_0<a_1 + a_2$.  We summarize these known results in the following
\begin{prop} \label{wpp} Each weighted projective plane $\C P^2_{a_0, a_1, a_2}$  with weights satisfying $a_0 \ge a_1 \ge a_2$ and $a_0 < a_1+a_2$ admits a conformally K\"ahler, Einstein--Maxwell metric $\tilde g$ which is Einstein. The corresponding conformal K\"ahler metric $g_{BF}= f^2\tilde g$ is the extremal Bochner-flat metric of $\C P^2_{a_0, a_1, a_2}$ and the Killing potential $f$ is a positive multiple of the scalar curvature $s_{BF}$ of $g_{BF}$.
\end{prop}
We now point out that the converse is also true, thus providing a uniqueness statement for these spaces.
\begin{thm}\label{wpp-classification} Any conformally K\"ahler,  Einstein--Maxwell metric $\tilde g$ on a weighted projective plane $\C P^2_{a_0, a_1, a_2}$  must be given by Proposition~\ref{wpp}. In particular,  if $a_0 \ge a_1 \ge a_2$ and $a_0 \ge a_1+a_2$, then $\C P^2_{a_0, a_1, a_2}$ does not admit any conformally K\"ahler,  Einstein--Maxwell metric, whereas if $a_0 \ge a_1 \ge a_2$ and $a_0 < a_1+a_2$ it admits a unique such metric up to isometry and homothety.
\end{thm}
\begin{proof} Let $\tilde g= \frac{1}{f^2}g$ be a conformally K\"ahler, Einstein--Maxwell metric on $(M,J)=\C P^2_{a_0, a_1, a_2}$. As pointed out in Section~\ref{s:preliminaries}, when $m=2$ the $2$-form $\psi= \frac{1}{f^2}\rho_0^{\tilde g}$ is anti-self-dual and co-closed (and therefore harmonic). As $b_2^-(M)=0$, it follows that $\psi \equiv 0$, i.e. ${\rm Ric}^{\tilde g}_0\equiv 0$. Thus, $\tilde g$ must be Einstein.  As proved in \cite{De}, the K\"ahler metric $g$ must be then extremal and $f = \lambda s_g$. By the uniqueness (modulo isometry) of the extremal (toric) metrics in their K\"ahler class~\cite{Guan}, we can assume that $g= g_{BF}$ is the  Bochner--flat metric on $(M,J)$ (see \cite{Bryant,Webster}),  and $f=s_{BF}$. As $f>0$ by assumption, $g_{BF}$ must have positive scalar curvature, i.e. the weights must satisfy (see e.g. \cite{Bryant,DG})  $a_0 \ge a_1 \ge a_2$ and $a_0 < a_1+a_2$.\end{proof}

\begin{rem}\label{r:wpp}
In terms of the theory developed in Section~\ref{s:toric}, $\C P^2_{a_0, a_1, a_2}$  is a symplectic toric orbifold, which up to a homothety of the symplectic form, is described  by  the standard simplex $\Delta \subset \R^2$ and labels ${\bf L}^{\bf a}=\{L_0=a_1a_2(1-\mu_1-\mu_2), L_1=a_0a_2 \mu_1, L_2=a_0a_1 \mu_2\}$, see \cite{abreu,ACGT2}.  Theorem~\ref{wpp-classification} tells us that \eqref{modified-abreu}  admits a solution for some positive affine function $f$ if and only if  the {\it extremal affine function} $\varepsilon_{(\Delta, {\bf L}^{\bf a})}$ of $(\Delta, {\bf L}^{\bf a})$ (see \cite{Do-02,legendre} for a definition) is everywhere positive,   and $f$ is a positive multiple of $\varepsilon_{(\Delta,{\bf L}^{\bf a})}$.
\end{rem}

\subsection{Toric orbifolds with $b_2=2$} We consider here the case of a compact toric symplectic $4$-orbifold  $(M, \omega, \T)$ for which the associated labelled Delzant polytope $(\Delta, {\bf L})$ is a convex compact quadrilateral in $\R^2$. Equivalently, we require that the second Betti number of $M$ equals $2$.  The existence of an extremal K\"ahler metric on such an orbifold has been solved (explicitly) in \cite{ambitoric2}, following ideas from \cite{legendre} (where, in particular,  the constant scalar curvature case has been classified). We shall follow \cite{ambitoric2} closely in order to establish Conjecture~\ref{c:donaldson},  demonstrate an explicit  method for deciding whether or not a $(\Delta, {\bf L}, f)$  is $K$-polystable,  and, in the latter case, construct explicitly the toric conformally-K\"ahler, Einstein--Maxwell metric associated to $(\Delta, {\bf L}, f)$.  For the convenience of the reader, we decided to provide a brief review of the methods  of \cite{ambitoric1,ambitoric2},  sending the reader to these references for  further details.

\subsubsection{Ambitoric structures}\label{s:ambitoric} We start by a  brief review of the local construction of toric metrics for $m=2$,  via the ambitoric ansatz of \cite{ambitoric1}.
\begin{defn} An \emph{ambik\"ahler structure} on a real $4$-manifold or
orbifold $M$ consists of a pair of K\"ahler metrics $(g_+, J_+, \omega_+)$ and
$(g_-, J_-, \omega_-)$ such that
\begin{bulletlist}
\item $g_+$ and $g_-$ induce the same conformal structure (i.e., $g_- =
\varphi^2g_+$ for a positive function $\varphi$ on $M$);
\item $J_+$ and $J_-$ have opposite orientations (equivalently the
volume elements $\frac1{2}\omega_+\wedge\omega_+$ and
$\frac1{2}\omega_-\wedge\omega_-$ on $M$ have opposite signs).
\end{bulletlist}
The structure is said to be \emph{ambitoric} if in addition
\begin{bulletlist}
\item there is a $2$-dimensional subspace $\tor$ of vector fields on $M$,
linearly independent on a dense open set, whose elements are hamiltonian and
Poisson-commuting Killing vector fields with respect to both $(g_+, \omega_+)$
and $(g_-, \omega_-)$.
\end{bulletlist}
\end{defn}
Thus $M$ has a pair of conformally equivalent but oppositely oriented K\"ahler
metrics, invariant under a local $2$-torus action, and both locally toric with
respect to that action.
There are three classes of examples of ambitoric structures.
\subsubsection{Toric products}\label{prod} Let $(\Sigma_1,g_1,J_1,\omega_1)$ and
$(\Sigma_2,g_2,J_2,\omega_2)$ be (locally) toric complex $1$-dimensional K\"ahler
manifolds or orbifolds, with hamiltonian Killing vector fields $K_1$ and
$K_2$. Then $M=\Sigma_1\times \Sigma_2$ is obviously ambitoric, with $g_\pm:=g=g_1\oplus
g_2$, $J_\pm=J_1\oplus (\pm J_2)$, $\omega_\pm=\omega_1\oplus (\pm\omega_2)$
and $\tor$ spanned by $K_1$ and $K_2$.  Writing the metrics of the toric Riemannian surfaces $(\Sigma_1,g_1)$ and  $(\Sigma_2, g_2)$ as
$$g_1 =  \frac{dx^2}{A(x)} + A(x)dt_1^2; \ \ g_2=\frac{dy^2}{B(y)} + B(y) dt_2^2,$$
for positive functions $A,B$ of one variable, and momentum/angular  coordinates
$$\mu_1=x, \ \mu_2=y,  \ t_1,  \ t_2,$$
the metric $g$ becomes
\begin{equation}\label{toric-product}
g = \frac{dx^2}{A(x)} +\frac{dy^2}{B(y)} +  A(x)dt_1^2 + B(y) dt_2^2,
\end{equation}
so that it is of the form \eqref{toricmetric} with corresponding matrix valued function ${\bf H} ={\bf H}^{A,B}$ given by
\begin{equation}\label{solution-parallelogram}
{\bf H}^{A,B} ={\rm diag} (A(x), B(y)).
\end{equation}
The scalar curvature of $g$  is $s=-(A''(x)+ B''(y)) = s_{g_1} + s_{g_2}$, showing that extremal metrics are given by taking $A$
and $B$ to be polynomials of degrees $\le 3$ (i.e. each $g_1$ and $g_2$ be extremal),  and CSC ones correspond to taking  $A$ and $B$ be polynomials of degrees $\le 2$ (i.e. $g_1$ and $g_2$ are CSC).

For a given (positive) affine function $f(x)= f_1x + f_0$ of $x$, say, the condition that the scalar curvature $s_{\tilde g}$ of $\tilde g =\frac{1}{f^2}g$ is  a constant $c$,  has been analyzed in \cite{LeB0}. In our notation,  one has by \eqref{sg} and \eqref{solution-parallelogram}
\begin{equation*}
\begin{split}
\frac{s_{\tilde g}}{f^5(x)}  &= -\Big(\frac{A(x)}{f^3(x)}\Big)_{xx} - \Big(\frac{B(y)}{f^3(x)}\Big)_{yy} \\
                      &=- \frac{1}{f^5(x)}\Big((A(x), f^2(x))^{(2)} - B''(y)f^2(x)\Big),
                      \end{split}
\end{equation*}
where
\begin{equation}\label{transvectant}
(G(x), R(x))^{(2)}: = R''(x) G(x) -3 R'(x)G'(x) + 6 R(x)G''(x).
\end{equation}
computes the transvectant of order $2$ when $R(x)$ is polynomial of degree $\le 4$ and $G(x)$ is a polynomial of degree  $\le 2$, see \cite[App.~A]{ambitoric1}. It follows that $s_{\tilde g}$ is  constant  iff and only if $A(x)= \sum_{k=0}^4 a_k x^k$  is a polynomial of degree $\le 4$, $B(y)= \sum_{k=0}^{2} b_k y^k$ is a polynomial of degree $\le 2$,  linked through the relation
\begin{equation}\label{EM-parallelogram}
 -(A(x), f^2(x))^{(2)}  = -2b_2f^2(x) +c.
 \end{equation}
 Note that in the case when $f$ is constant we obtain ${\rm deg}(A) \le 2$ (i.e. the metric $g$ is CSC) whereas in the case when $f(x) = f_1x + f_0$ with $f_1\neq 0$, up to an affine transformation of the variable $x$ and rescaling of $f$,  we can assume  $f(x)=x$. Then,  \eqref{EM-parallelogram} is equivalent to the constraints  (compare with \cite{LeB0}):
 \begin{equation}\label{EM-parallelogram-normal}
 a_2 =-b_2, \ \ a_1=0.
 \end{equation}

\subsubsection{Toric Calabi-type metrics}\label{s:calabi} Let $(\Sigma,g_\Sigma,J,\omega_\Sigma)$ be a locally toric
complex $1$-dimensional K\"ahler manifold,  with a hamiltonian Killing vector
field $V$ and  momentum $y$.  Let $\pi\colon P\to \Sigma$ be a (local) $\R$-bundle with connection
$\theta$ and curvature $d\theta=\pi^*\omega_{\Sigma}$, and  $A(x)$ be a positive
function defined on an open interval $(\al_\1,\al_\2) \subset \R^+$. Then on $M=P\times I$,  the metric
\begin{align}\label{calabi-type}
g &= x g_{\Sigma}+\frac{xdx^2}{A(x)}+ \frac{A(x)}{x}\theta^2, \ d\theta = \omega_{\Sigma}\\
\omega &= x \omega_{\Sigma}  + dx\wedge\theta,\qquad
J (x dx) =  A(x) \theta,
\end{align}
is K\"ahler, and the generator $K$ of the $\R$
action on $P$ along with the lift $\tilde V = V^H + y K$ of the hamiltonian Killing field of $(\Sigma, g_{\Sigma}, \omega_{\Sigma})$ to $M$ define a local toric structure. It is not hard to see that $(g_+=g, \omega_+=\omega)$  is ambitoric with associated conformal K\"ahler metric $g_- = \frac{1}{x^2} g_+, \omega_-= \frac{1}{x} \omega_{\Sigma} - \frac{1}{x^{2}} dx\wedge\theta$. It follows that $g_-$ is again given by \eqref{calabi-type}, by  replacing $x$ with $\bar x = \frac{1}{x}$ and  the function $A(x)$ with $A^*(\bar x)=\bar{x}^4 A(1/\bar{x})$.
Writing the toric metric $(g_{\Sigma}, \omega_{\Sigma})$ in momentum/angle coordinates with respect to $V$ as
\begin{equation}\label{FS}
g_{\Sigma}= \frac{{dy}^2}{B(y)} + B(y) dt_2^2, \ \omega_{\Sigma}  = dy \wedge dt_2,
\end{equation}
for a positive function $B(y)$,  the K\"ahler metric $(g, \omega)$ becomes (see \cite{legendre})
\begin{equation}\label{calabi-type-toric}
\begin{split}
g &= x \frac{dx^2}{A(x)} + x \frac{dy^2}{B(y)} + \frac{A(x)}{x}(dt_1 + y dt_2)^2 + x B(y) dt_2^2 \\
\omega & = dx \wedge dt_1 + d(xy) \wedge dt_2,
\end{split}
\end{equation}
where
\begin{equation}\label{calabi-moments}
(\mu_1, \mu_2)=(x, xy)
\end{equation}
are the momentum coordinates and $(t_1,t_2)$ are angular coordinates of $(g,\omega)$ with respect to $(K, \tilde V)$. The corresponding matrix valued function ${\bf H} = {\bf H}^{A,B}$ is then
\begin{equation}\label{solution-trapezoid}
{\bf H}^{A,B} = \frac{1}{x}\Big(\begin{array}{cc} A(x) & yA(x) \\  yA(x) & x^2B(y) + y^2A(x) \end{array} \Big),
\end{equation}
The ansatz \eqref{calabi-type} has been used in many places to construct explicit examples of extremal K\"ahler metrics.  We follow here the notation in \cite{legendre}. One can compute the scalar curvature $s_{\tilde g}$ of $\tilde g= \frac{1}{f^2} g$ from \eqref{sg} and \eqref{solution-trapezoid}, by using
\begin{equation*}
\frac{\partial}{\partial \mu_1} = \frac{\partial}{\partial x} - \frac{y}{x} \frac{\partial}{\partial y}, \ \ \frac{\partial}{\partial \mu_2} = \frac{1}{x} \frac{\partial}{\partial y}.
\end{equation*}
Let us consider the special case when $f(x)=f_0 + x f_1$ is an everywhere positive function depending only on $x$.  We then get
\begin{equation*}
\begin{split}
\frac{s_{\tilde g} }{f^5} &= -\Big(\frac{A(x)}{xf^3}\Big)_{xx} - 2\Big(\frac{A(x)}{x^2 f^3}\Big)_x - \frac{B''(y)}{xf^3} - \frac{2A(x)}{x^3f^3}
\end{split}
\end{equation*}
The case $f=f_0=const$ is analyzed in \cite{legendre}: the metric $\tilde g$ is cscK and this is equivalent to $B(y)= b_2y^2 + b_1y  + b_0$ being a polynomial of degree $\le 2$ (i.e. $(\Sigma,g_\Sigma)$ having constant Gauss curvature $s_{\Sigma} = -2b_2$) and $A(x)= \sum_{k=0}^4 a_k x^k$ being  a polynomial of degree $\le 3$ (so $a_4=0$) with coefficient $a_2 = -b_2$. In this case $s_{\tilde g} = -\frac{6a_3}{f_0}.$

We now assume (without loss) that $f(x)= x + \al$ with $x + \al >0$. Then,   see \cite{LeB}, $\tilde g$ has constant scalar curvature  if and only if  $B(y)= b_2y^2 + b_1y  + b_0$ is a polynomial of degree $\le 2$ and $A(x-\al)=\sum_{k=0}^4 a_k x^k$  is a polynomial of degree $\le 4$,  such that
\begin{equation}\label{EM-calabi-normal}
a_2=-b_2, \ 2a_0= \alpha a_1.
\end{equation}

\subsubsection{Regular ambitoric structures}\label{s:regular} Let $q(z)=q_2z^2+2q_1 z+ q_0$ be a
quadratic polynomial and $M$ a $4$-dimensional manifold (or orbifold)
with real-valued functions $(x,y,\tau_0,\tau_1,\tau_2)$ such that $x>y$,
$2q_1\tau_1=q_2\tau_2+q_0\tau_0$, and such that $dx, dy, d\tau_0, d\tau_1, d\tau_2$ span the
cotangent space of $M$. Let $\tor$ be the $2$-dimensional space of vector fields $K$
on $M$ with $dx(K)=0=dy(K)$ and $d\tau_j(K)$ constant, and let $A$ and $B$ be
functions on open neighbourhoods of the images of $x$ and $y$ in $\R$.

Then, on a suitable open sub-manifold $M^0 \subset M$,   we define an ambitoric metric with respect to $\tor$ by
\begin{align} \label{g-pm}
g_\pm =  \biggl(\frac{x-y}{q(x,y)}\biggr)^{\pm1}
\biggl(& \frac{dx^2}{A(x)} + \frac{dy^2}{B(y)}  + A(x) \Bigl(\frac{y^2 d\tau_0 + 2y d\tau_1 + d\tau_2}{(x-y)q(x,y)}\Bigr)^2 \\ \nonumber
& + B(y) \Bigl(\frac{x^2 d\tau_0 + 2x d\tau_1 + d\tau_2}{(x-y)q(x,y)}\Bigr)^2
\biggr),\\
\label{omega-pm}
\omega_\pm = \biggl(\frac{x-y}{q(x,y)}\biggr)^{\pm 1}
& \frac{dx\wedge (y^2 d\tau_0 + 2y d\tau_1 + d\tau_2)
\pm dy \wedge (x^2 d\tau_0 + 2x d\tau_1 + d\tau_2)}{(x-y)q(x,y)},\\ \nonumber
J_\pm dx= A(x)&\frac{y^2 d\tau_0 + 2y d\tau_1 + d\tau_2}{(x-y)q(x,y)}, \quad
J_\pm dy= \pm B(y)\frac{x^2 d\tau_0 + 2x d\tau_1 + d\tau_2}{(x-y)q(x,y)},
\end{align}
where, for any polynomial $q(z)=q_2z^2 + 2q_1z + q_0$ of degree $\le 2$, we let $q(x,y):=q_2xy+q_1(x+y)+q_0$, and the functions $A(x)$ and $B(y)$  are assumed positive on $M^0$.

Denote by $S^2$ the $3$-dimensional vector space of polynomials of degree $\le 2$. The discriminant, viewed as a natural quadratic form on $S^2$, gives rise to an inner product $\langle \cdot, \cdot \rangle$.  The space of Killing fields of $g_{\pm}$  for  the torus is naturally isomorphic to the space
\begin{equation*}
S^2_{0,q} =\{p(z) \in S^2 :  \langle p, q \rangle =2p_1 q_1 -(q_2p_0 + q_0p_2)=0.\}
\end{equation*}
The space $S^2_{0,q}$ in turn  is isomorphic to the quotient space $S^2 /\langle q \rangle$ using  the map $\frac{1}{2} {\rm ad}_q$, where
\begin{equation*}
{\rm ad}_q(w)=\{q,w\}= q'w - w'q
\end{equation*} is the Poisson bracket on $S^2$. Thus, if $(p_1,p_2)$ is a basis of $S^2_{q,0}$ and $(w_1,w_2)$ the corresponding basis of $S^2 /\langle q \rangle$ (with $p_i = \frac{1}{2}\{q, w_i\}$), momentum/angular coordinates for $g_{\pm}$ are given by
\begin{equation}\label{moments}
\begin{split}
\mu_i^+ &= w_i(x,y)/q(x,y), \ t_i, \  (i=1,2) \\
\mu_i^- & = p_i(x,y)/(x-y), \ t_i, \  (i=1,2).
\end{split}
\end{equation}
It is straightforward to compute that the toric forms \eqref{toricmetric} of the metrics $g_{\pm}$ are given by matrix-valued smooth functions  ${\mathbf H}^{A,B}_{\pm}$, respectively,
\begin{equation}\label{solution-regular}
\begin{split}
{\mathbf H}^{A,B}_{-}(p_i,p_j)&=\frac{A(x) p_i(y)  p_j(y) + B(y) p_i(x) p_j(x)}{(x-y)^3\, q(x, y) },\\
{\mathbf H}^{A,B}_{+}(p_i,p_j)&=\frac{A(x) p_i(y) p_j(y) + B(y) p_i(x) p_j(x)}{(x-y)\, q(x, y)^3 }.
\end{split}
\end{equation}

As shown in \cite[Sect.~5.2]{ambitoric1},  the metric $g=g_+$ is extremal iff $g_-$ is extremal  iff
\begin{equation}\label{regular-extremal}\begin{split}
A(z)&=q(z)\pi(z)+P(z),\\
B(z)&=q(z)\pi(z)-P(z),\\
\end{split}\end{equation}
where $\pi(z)=\pi_2z^2 + 2\pi_1z + \pi_0$ is a polynomial of degree at most two satisfying  $\langle \pi, q \rangle=0$,  and $P(z)$ is
polynomial of degree at most four.  In this case,  the scalar curvature is given by
\begin{equation}\label{ambitoric-s}
s_g =  -\frac{\{q, (q, P)^{(2)}\}(x,y)}{q(x,y)},
\end{equation}
where $(q, P)^{(2)} (x)$ is the polynomial of degree $\le 2$ given by \eqref{transvectant} and, we recall, $\{\cdot, \cdot \}$ denotes the Poisson bracket of elements of $S^2$.

Similarly,  if $p(z)=p_2z^2 + 2p_1z + p_0$ is a polynomial of degree $\le 2$ such that  $\langle q, p \rangle=0$ and $p(x,y)>0$ on $M_0$,  it is shown in \cite[Sect.~5.2]{ambitoric1} that  the metric $\tilde g= \Big(\frac{q(x,y)}{p(x,y)}\Big)^2 g_+$ is conformally-K\"ahler,  Einstein--Maxwell provided that
\begin{equation}\label{regular-EM}\begin{split}
&A(z)=p(z)\rho(z)+R(z),\\
& B(z) =p(z)\rho(z)-R(z),\\
& \langle (p, R)^{(2)}, q \rangle =0
\end{split}\end{equation}
where $\rho(z)=\rho_2z^2 + 2\rho_1z + \rho_0$ is a polynomial of degree at most two, satisfying $\langle \rho, p \rangle =0$,  and $R(z)$ is
polynomial of degree at most $4$.

\begin{rem}\label{r:ambitoric-scalar}
It is straightforward to check that if $g=g_+$ is given by \eqref{g-pm}  for some polynomials $q(z)$, $A(z)$ and $B(z)$ such that $A(z)$ and $B(z)$ satisfy the first two conditions of \eqref{regular-EM} for some polynomials $p(z)$ and $\rho(z)$ of degree $\le 2$, with $\langle p, \rho \rangle =0$,  and some polynomial $R(z)$ of degree $\le 4$, then  the scalar curvature $s_{\tilde g}$ of $\tilde g= \big(\frac{q(x,y}{p(x,y)}\big)^2 g$  is equal to
\begin{equation}\label{ambitoric-sg}
s_{\tilde g} =  -\frac{\{p, (p, R)^{(2)}\}(x,y)}{q(x,y)}.
\end{equation}
This is,  in general,  an affine function in momenta, i.e. the ansatz produces an $f$-extremal metric in the terminology of Sect.~\ref{s:extremal}, where $f$ is the affine function in momenta given in the $(x,y)$-coordinates by  $\frac{p(x,y)}{q(x,y)}$.  If we assume, moreover, that $\langle p, q \rangle =0$ and $\langle (p, R)^{(2)}, q \rangle =0$, elementary algebra shows that $\{p, (p, R)^{(2)}\}= c q$ for some constant $c$, or equivalently, $s_{\tilde g}=c$ is a constant.
\end{rem}

\begin{defn}\label{rough} A (local) K\"ahler metric $(g, \omega)$  will be called ambitoric of {\it product}, {\it Calabi } or {\it postive/negative regular} type  if it is given by \eqref{toric-product}, \eqref{calabi-type-toric}, or one of the metrics $(g_\pm, \omega_{\pm})$ defined by \eqref{g-pm} with respect to some $q(z)$, respectively,  for some positive functions $A(x)$, $B(y)$. We will respectively refer to the corresponding examples of extremal K\"ahler metrics (resp. conformal Einstein--Maxwell metrics $\tilde g = \frac{1}{f^2} g$) as {\it ambitoric extremal} (resp. {\it ambitoric Einstein--Maxwell}) metrics of product, Calabi or positive/negative regular type.
\end{defn}

The local discussion above leads to the following useful observation.

\begin{prop}\label{duality}There is a bijective correspondence between ambitoric Einstein--Maxwell metrics $\tilde g= \frac{1}{f^2} g$ of  positive regular type,  associated to a positive Killing potential $f(x,y)$ on one side,  and ambitoric extremal K\"ahler metrics  $\bar g$  of positive regular type, of positive scalar curvature $s_{\bar g}(\bar x,\bar y)$ on the other, such that $f(x,y) = 1/s_{\bar g}(x,y)$.  Furthermore,  if  $q(z), A(z), B(z)$ are the polynomials defining $g$, and $p(z)$ is the quadric polynomial defining the Killing potential $f =p(x,y)/q(x,y)$, then $\bar g$ has corresponding polynomials $\bar q (z)= p(z), \bar A(z)=A(z), \bar{B}(z)=B(z)$ whereas $q(x,y)= s_{\bar g}(x,y) \bar{q}(x,y)$. In particular, $g= \bar g$  if and only if $f=const$, i.e. $\tilde g$ is a cscK regular ambitoric metric.
\end{prop}

\subsubsection{Ambitoric compactifications}\label{s:compactification}
We now recall from the theory in \cite{ambitoric2} how the local considerations in the previous subsections lead to ambitoric
examples on compact toric orbifolds.

The main observation is than in each case (product, Calabi, or regular) any line $x= \al$ (resp. $y=\be$)  in the $(x,y)$-coordinates transforms to an affine line $\ell_{\al}=0$ (resp. $\ell_{\be}=0$) in the momentum $(\mu_1, \mu_2)$-coordinates. Indeed, this is obvious in the product case when we take $\ell_{\al} = \al -\mu_1$ (resp. $\ell_{\be}= \be -\mu_2$) and in the Calabi case if we have $\ell_{\alpha} = -\al \mu_1 +\al^2$ (resp. $\ell_{\be} = \be \mu_1 - \mu_2$), whereas in the regular case  we use the identification between $\tor$ and $S^2_{q,0}$ to write $\ell^+_{\alpha}=(x-\alpha)(y-\alpha)/q(x,y)=0$ (resp. $\ell^+_{\beta}=(x-\beta)(y-\beta)/q(x,y)$) in the $(\mu_1^+, \mu_2^+)$-plane, and similarly, $\ell^-_{\alpha}=(x-\alpha)q(y, \alpha)/(x-y)$ (resp. $\ell_{\beta}^-=(y-\beta)q(x,\beta)/(x-y)$)  in the $(\mu_1^-, \mu_2^-)$-plane.

In each case, we denote by $p_{\al}=d\ell_{\alpha}$ and $p_{\be}=d\ell_\be$ the corresponding normals: in the product case we have $p_{\al}=(-1,0), p_{\be}=(0,-1)$, in the Calabi case $p_{\al} =(-\al, 0), p_{\be}=(\be, -1)$ whereas in the regular case
\begin{equation}\label{normals}
p^{\pm}_{\alpha}(z) = q(\alpha, z)(z-\alpha) ; \ \ p^{\pm}_{\be} (z) = q(z, \be)(z-\be),
\end{equation}
viewed as elements of $S^2_{0,q}$.

Using the freedom in the choice of coordinates $(x,y)$ (in the product and Calabi cases $(x,y)$ can be chosen up to a common additive factor whereas in the regular case $(x,y)$ are determined up to the action of ${\rm SL}(2, \R)$), we can assume without loss  that $x\in [\al_\1, \al_\2]$, $y\in [\be_\1, \be_\2]$ where  $0<\be_\1<\be_\2<\al_\1<\al_\2$ and  $q(x,y)>0$ on $[\al_\1, \al_\2] \times [\be_\1, \be_\2]$ in the regular case. This data then determine a labelled quadrilateral $(\Delta, {\bf L})$ with  labelling ${\bf L}= (\frac{1}{r_{\al, \1}}\ell_{\al_\1}, \frac{1}{r_{\al,\2}}\ell_{\al_\2}, \frac{1}{r_{\be, \1}}\ell_{\be_\1}, \frac{1}{r_{\be, \2}}\ell_{\be_\2})$, where the real numbers  $r_{\al,k}, r_{\be, k}$ must satisfy
\begin{equation*}
r_{\al,\1} < 0 < r_{\al,\2}, \  r_{\be,\1} <  0 < r_{\be,\2}
\end{equation*}
in order for ${\bf L}$ to determine a convex compact quadrilateral $\Delta$. (Here we have simplified notation in the
regular case, writing $\ell$ rather than $\ell^\pm$ when the meaning is clear.)

Using the expressions \eqref{solution-parallelogram}, \eqref{solution-trapezoid}, \eqref{solution-regular} for the corresponding matrices ${\bf H}={\bf H}^{A,B}$,  one can show that the boundary conditions given  in  Proposition~\ref{p:toric-ccs} are equivalent to
\begin{equation}\label{ambitoric-boundary}
A(\al_k)=0=B(\be_k)=0, \, A'(\al_k)=-2r_{\al,k}, \, B'(\be_k)
=2r_{\be,k} \quad (k=\1,\2),
\end{equation}
whereas the positivity and smoothness conditions mean that $A(x)$ (resp. $B(y)$) is smooth and positive on $[\al_\1, \al_\2]$ (resp. on $[\be_\1, \be_\2]$). This leads to the following (compare with \cite[Prop.~3]{ambitoric2})
\begin{prop}\label{ambitoric-compactification}  Given the data of real numbers
\begin{equation*}
0<\be_\1<\be_\2<\al_\1<\al_\2, \ r_{\al,\1} < 0 < r_{\al,\2}, \  r_{\be,\1} <  0 < r_{\be,\2}
\end{equation*}
which satisfy the integrality condition
\begin{equation*}
{\rm span}_{\Z}\Big\{\frac{p_{\al_\1}}{r_{\al,\1}}, \frac{p_{\al_\2}}{r_{\al,\2}}, \frac{p_{\be_\1}}{r_{\be,\1}}, \frac{p_{\be_\2}}{r_{\be, \2}} \Big\} \cong \Z^2,
\end{equation*}
 two smooth functions $A(z)$ and $B(z)$ satisfying  the positivity conditions $A(x)>0$ on $(\al_\1, \al_\2)$, $B(y)>0$ on  $(\be_\1, \be_\2)$ and the boundary conditions \eqref{ambitoric-boundary}, and, in the regular case,   a quadric $q(z)$  satisfying  $q(x,y) >0$ on $[\al_\1, \al_\2] \times [\be_\1, \be_\2]$, the corresponding ambitoric metric compactifies as a toric K\"ahler metric on the compact toric orbifold $M$ classified by the labelled quadrilateral  $(\Delta, {\bf L})$.
\end{prop}
\begin{defn} We shall refer to the compact K\"ahler orbifold metrics obtained in Proposition~\ref{ambitoric-compactification} as {\it ambitoric compactifications}, and additionally describe them as positive/negative if we are referring specifically to
the metric of positive/negative regular type. Note that the quadrilaterals corresponding to ambitoric compactifications
of product and Calabi types are, respectively, parallelograms and trapeziods.
\end{defn}
\subsubsection{Classification of ambitoric  Einstein--Maxwell metrics}
We are now ready to prove our main result of this section.
\begin{thm}\label{ambitoric-EM-classification} Let $(M, \omega, \T)$ be a compact symplectic toric $4$-orbifold whose rational Delzant polytope is a labelled quadrilateral $(\Delta, {\bf L})$ {\rm (}equivalently,  $b_2(M)=2${\rm )} and $f$ the pull-back to $M$ via the moment map of a strictly positive affine function on $\Delta$. Then the following conditions are equivalent
\begin{enumerate}
\item[(i)] $(M, \omega, \T)$ is $K$-polystable with respect to $f$;
\item[(ii)] $(M,\omega, \T)$ admits a $\T$-invariant K\"ahler metric $g$ such that $\tilde g = \frac{1}{f^2} g$ is  a conformally K\"ahler, Einstein--Maxwell metric on $M$;
\item[(iii)] $(M,\omega, \T)$ admits an ambitoric compactification $g$ such that $\tilde g = \frac{1}{f^2} g$  is an Einstein--Maxwell  metric on $M$.
\end{enumerate}
In particular, assuming condition (iii),  if $(M,\omega)$ admits any $\T$-invariant conformally-K\"ahler, Einstein--Maxwell metric with Killing potential $f$, it must be an ambitoric compactification.
\end{thm}
\begin{proof} Using  Corollary~\ref{c:K-stability}, ${\rm (iii)} \Rightarrow {\rm (ii)} \Rightarrow {\rm (i)}$. Also, by the uniqueness (Theorem~\ref{uniqueness}), assuming condition (iii), any $\T$-invariant conformally-K\"ahler, Einstein--Maxwell metric with Killing potential $f$ on $(M,\omega)$ must be an ambitoric compactification.

We need to show ${\rm (i)} \Rightarrow {\rm (iii)}$.  We follow very closely the  proof of \cite[Thm.~ 1]{ambitoric2} which, in turn, is a generalization of the method in \cite{legendre}.

Let us denote by $\mathfrak{h}= {\rm Aff}(\tor^*)$ the $3$-dimensional vector space of affine functions in momenta. There is a distinguished element $\boldsymbol{1} \in \mathfrak{h}$, representing the constant affine function equal to $1$. We consider the pencil of conics on $\Proj(\mathfrak{h})$ which pass through the four points of $\Proj(\mathfrak{h})$ corresponding to the four affine functions defining ${\bf L}$.
There exists a unique conic in this pencil,  whose associated symmetric bilinear form makes $\boldsymbol{1}$ and $f$ orthogonal. We shall denote this conic by $\mathcal{C}(\Delta,  f)$, noting that  $\mathcal{C}(\Delta,  f)$ does not depend on rescaling the elements of  ${\bf L}$.

\noindent
{\bf Case 1.}  $\mathcal{C}(\Delta, f)$ is non-singular.  As $f>0$ on $\Delta$ (and therefore the affine line $f=0$ does not intersect the  segment of the Newton line between midpoints of the diagonals of $\Delta$), by \cite[Prop.~4 \& 5]{ambitoric2},  $(\Delta, {\bf L})$ admits a positive regular ambitoric compactification such that, using the identification of the space of quadric polynomials $S^2$ with $\mathfrak{h}$ which sends an element $w\in S^2$ to the affine function in momenta  $w(x,y)/q(x,y)$, $\mathcal{C}(\Delta, f)$ is identified  with the (polarization of) the discriminant of elements of $S^2$, i.e.  the inner product $\langle \cdot, \cdot \rangle$ for elements in $S^2$. It follows that $\boldsymbol{1}$  corresponds to $q$   and  $f= \frac{p(x,y)}{q(x,y)}$ for a polynomial $p$ such that  $\langle p, q \rangle =0$.

Fixing the data $\al_{k}, \be_k, r_{\al,k}, q$ of the ambitoric compactification, and varying the functions $A(z)$ and $B(z)$,  introduces a class of compatible positive ambitoric compactifications of $(\Delta, {\bf L}, f)$.  According to Proposition~\ref{ambitoric-compactification}, in order to show the existence of an ambitoric compactification $g$ such that $\tilde g = \frac{1}{f^2} g$  is an Einstein--Maxwell  metric on $M$, it suffices  to find polynomials $A(z)= \sum_{k=0}^4 a_k z^k$ and $B(z)= \sum_{k=0}^4 b_k z^k$ of degree $\le 4$, which satisfy  \eqref{regular-EM}, the boundary conditions \eqref{ambitoric-boundary}, and which are positive respectively on $(\al_\1, \al_\2)$ and $(\be_\1, \be_\2)$. Leaving momentarily the positivity condition for $A$ and $B$ aside, we now make a count of the free
constants versus the constraints in this problem. First, as $q$ is fixed, and as $f$ is also given, the quadric $p$ is also fixed (recall that any affine function in momenta is written as $f= \frac{p}{q}$). Thus,  the first two
relations in \eqref{regular-EM} show that the polynomials $A$, $B$ are determined by eight constants (the five coefficients of $R$, along with the three coefficients of $\rho$). These constants are subject to two linear relations (the condition that $\rho$ and $p$ are orthogonal and the last relation in \eqref{regular-EM}) which will be viewed as  two linear constraints on these eight constants. On the other hand, it is easy to see that the boundary conditions \eqref{ambitoric-boundary} (for fixed $\al_k, \be_k$) determine  uniquely the $8$ constants as linear functions of $r_{\al,k}$ and $r_{\be, k}$. Thus, the two linear constraints can be re-expressed as two linear relations between the real numbers $r_{\al, k}$ and $r_{\be,k}$.

Note that if polynomials $A(z), B(z)$ satisfying  \eqref{regular-EM} and \eqref{ambitoric-boundary} exist, then,  even if they do not satisfy the positivity condition, they still give rise to a smooth matrix ${\bf H}^{A,B}$ satisfying \eqref{matrix-abreu} and the boundary conditions \eqref{toricboundary}
on $(\Delta, {\bf L})$. By Theorem~\ref{toric-futaki}, the Futaki invariant $\mathfrak{F}_{\Delta, {\bf L}, f}$ defined by \eqref{Donaldson-Futaki-1} must vanish, which is equivalent to
$$\int_{\partial \Delta} \frac{\varphi}{f^{3}}d\sigma \int_{\Delta} \frac{1}{f^{5} }d\mu =  \int_{\partial \Delta} \frac{1}{f^{3} }d\sigma  \int_{\Delta} \frac{\varphi}{f^{5}}d\mu$$
for $\varphi = \mu_i, i=1,2$. This in turn can also be expressed via two  linear constraints on the parameters $r_{\al,k}, r_{\be,k}$, as the latter appear as multiplication factors of $d\sigma$ restricted to the corresponding facet (see \eqref{boundary-measure}).   Rank considerations for the constraints corresponding to the vanishing of the Futaki invariant thus allow us to conclude that  the two sets of two linear conditions are in fact equivalent. Since the $K$-polystability of $(\Delta, {\bf L}, f)$ implies the vanishing  of $\mathfrak{F}_{\Delta, {\bf L}, f}$, it follows that there exist polynomials $A(z)$ and $B(z)$ satisfying \eqref{regular-EM} and the boundary conditions \eqref{ambitoric-boundary}.

We still need to show that $A(z)$ and $B(z)$ are positive, respectively, on $(\al_\1, \al_\2)$ and $(\be_\1, \be_\2)$.  To this end, suppose that $A$, say, vanishes at $x=\alpha$, $\alpha \in (\al_\1, \al_\2)$.  Consider the simple crease function $\varphi_{\alpha} = {\rm max}\{0, \ell_{\al}^+\}$ where, we recall from \S~\ref{s:compactification},  $\ell_{\al}^+= \frac{(x-\al)(y-\al)}{q(x,y)}$. As shown in the proof of Corollary~\ref{c:K-stability} (where positive definiteness of ${\bf H}$ is not used),  the formal solution ${\bf H}={\bf H}^{A,B}_+$ of \eqref{matrix-abreu} can be used to compute
\begin{equation*}
\mathfrak{F}_{\Delta, {\bf L}, f}(\varphi_{\alpha}) = \int_{F_{\al}} {\bf H}^{A,B}_+(p^+_{\al}, p^+_{\al})/f^{2m-1} d\sigma,
\end{equation*}
where $F_{\al}$ denotes the crease of $\varphi_\alpha$, i.e.  the segment in $\Delta$ on which $\ell^+_{\alpha}=0$ and $p^+_{\al}= d\ell^+_{\al}=q(\al, z)(z-\al)$ is the normal. By \eqref{solution-regular}, the latter integral must be zero, which is a contradiction with the $K$-polystability of $(\Delta, {\bf L}, f)$.  We thus conclude that $A(x)$ (resp. $B(y)$) does not vanish on $(\al_\1, \al_\2)$ (resp. $(\be_\1, \be_\2)$), so it must be  positive  according to the sign of its first derivative at $\al_\1$ (resp. $\be_\1$).

\smallskip
\noindent
{\bf Case 2.}   $\mathcal{C}(\Delta, f)$ is singular (i.e. a pair of lines) and $\boldsymbol{1} \notin \mathcal{C}(\Delta, f)$. Dualizing, this is the case when $\Delta$ is not a trapezoid and the affine line $\{f=0 \}$ passes through the mid-point of the exterior diagonal of $\Delta$ (see the proof of \cite[Prop. 5]{ambitoric2}). By \cite[Prop.~4 \& 5]{ambitoric2},  $(\Delta, {\bf L})$ admits a negative regular ambitoric compactification, associated to the  data $\al_k, \be_k,  r_{\al, k}, r_{\be, k}$ and $q$.  As observed in Sect.~\ref{s:regular}, in the corresponding $(x,y)$-coordinates, the affine function $f$ can be written as $f(x,y) = \frac{p(x,y)}{x-y}$ for a quadric $p$ satisfying $\langle p, q \rangle=0$.  The same argument as the one developed in Case 1 shows that one can find polynomials $A(z)$ and $B(z)$ which satisfy \eqref{regular-EM}, \eqref{ambitoric-boundary}, and which are positive respectively on $(\al_\1, \al_\2)$ and $(\be_\1, \be_\2)$, thus proving the claim.

\smallskip
\noindent
{\bf Case 3.} $\mathcal{C}(\Delta, f)$ is singular (i.e. a pair of lines) and $\boldsymbol{1} \in \mathcal{C}(\Delta, f)$. If ${\bf 1}$ is the singular point  of $\mathcal{C}(\Delta,f)$, i.e.  the intersection of the two lines defining $\mathcal{C}(\Delta, f)$,  the dual conic $\mathcal{C}^*(\Delta, f)$ is a double line at infinity, where the affine structure on $\Proj(h^*)$ is defined with
respect to the projective line in $\Proj(h^*)$ dual to ${\bf 1} \in \Proj(h)$. As this line must pass through the (exterior) points of intersection of the lines of opposite facets  of $\Delta$ (recall that $\mathcal{C}(\Delta,f)$ is in the pencil of conics
passing through the  $4$ points determined by the affine lines of ${\bf L}$), $\Delta$ must be a parallelogram. In this case,  the orthogonality of $\boldsymbol{1}$  and $f$  means that $f$ belongs to one of the pair of lines $\mathcal{C}(\Delta, f)$ determines, or equivalently, that the affine line $\{ f=0 \}$ is parallel to a pair of parallel facets of $\Delta$. Similarly,  if
${\bf 1}\in \mathcal{C}(\Delta, f)$ is not the singular point of $\mathcal{C}(\Delta, f)$, the dual conic $\mathcal{C}^*(\Delta, f)$ is a double line passing through the point at infinity, which must  then be the intersection point of two lines determined by a pair of opposite facets of $\Delta$, i.e.  $\Delta$ must be a trapezoid which is not a parallelogram. The orthogonality between $\boldsymbol{1}$ and $f$ translates again into the fact that the affine line  $\{f=0\}$ is  parallel to the pair of parallel facets of $\Delta$. One can thus use the product or Calabi type ambitoric metrics to obtain an ambitoric compactification on
$(\Delta, {\bf L})$. The fact that $\{f=0\}$ is  parallel to a pair of  parallel facets of $\Delta$ means that $f$ can be written as an affine function of $x$ (with respect to the ambitoric compactification). The case when $f$ is constant is analyzed in \cite{legendre}. By rescaling $f$ if necessary, we may assume $f(x)= x+ \al$, and if we can construct the ambitoric Einstein--Maxwell metrics using Proposition \ref{ambitoric-compactification}, they would have to be of product or Calabi type.
To this end, we need to find two polynomials $A(z)=\sum_{k=0}^4 a_k z^k$ and $B(z)= \sum_{k=0}^2 b_k z^k$ of degrees $\le 4$ and
$\le 2$, respectively, which satisfy \eqref{EM-parallelogram-normal} or \eqref{EM-calabi-normal}, accordingly to the type, as well as \eqref{ambitoric-boundary} and the positivity conditions. The important observation is that in either case (as
${\rm deg}(B) =2$), the four boundary conditions for $B(z)$ imply a first overdetermined linear constraint
$r_{\be,\1} =-r_{\be,\2}$, so in particular $B(y)$ is determined uniquely. Then, as $A(z)$ has five coefficients, the four boundary conditions for $A(z)$ together with the two linear constraints for its coefficients (given by \eqref{EM-parallelogram-normal} or
\eqref{EM-calabi-normal}, according to the type) determine altogether a second overdetermined linear constraint. This translates
again to two linear constraints for the real numbers $r_{\al, k}, r_{\be, k}$. The rest of the argument is identical to the one in Case~1. \end{proof}

\begin{rem}\label{extension} The proof of Theorem~\ref{ambitoric-EM-classification} shows that for any labelled quadrilateral $(\Delta, {\bf L})$ and a positive affine function $f$, the existence of  polynomials  $A(z), B(z)$ of degree $\le 4$,  which satisfy \eqref{regular-EM} and the boundary conditions \eqref{ambitoric-boundary} is equivalent to the vanishing of the Futaki invariant $\mathfrak{F}_{\Delta, {\bf L}, f}$. This observation can be used as an effective tool to compute whether or not  $\mathfrak{F}_{\Delta, {\bf L}, f}=0$.  On the other hand, there are cases when $\mathfrak{F}_{\Delta, {\bf L}, f}=0$ is not only necessary but also sufficient for the existence of a solution of the Einstein--Maxwell equation \eqref{modified-abreu}. Indeed, if $(\Delta, {\bf L})$ is a labelled trapezoid and the affine line in $\tor^*$ determined by $f$ is  parallel to a pair of parallel facets of $\Delta$, then the proof of Theorem~\ref{ambitoric-EM-classification} shows that vanishing of the Futaki invariant implies there exist a formal solution of product or Calabi type. In either case, the boundary conditions \eqref{ambitoric-boundary} and the respective linear constraints \eqref{EM-parallelogram-normal} and \eqref{EM-calabi-normal} for $A(z)$ and $B(z)$  imply that $A(x)$ and $B(y)$ are positive on their respective domains of definition (this is already observed in \cite[p.379]{legendre}): for $B(y)$ this is obvious as ${\rm deg} (B)=2$, $B(\beta_{\1})=0$ and $B'(\beta_{\1})>0$ (which  also forces $b_2<0$), whereas if $A(z)$ has all of its roots in the interval $[\alpha_{\1}, \alpha_{\2}]$, $A'(\alpha_{\1})>0$ then forces $a_4<0$, which is incompatible with the linear constraint $a_2=-b_2>0$ (recall that we assumed $\alpha_{\2}>\alpha_{\1} >0$).  Similarly, e.g. by using Proposition~\ref{duality} and an argument from \cite[Example~1]{ambitoric2}, one can show that the vanishing of $\mathfrak{F}_{\Delta, {\bf L}, f}$ is a sufficient condition for $K$-polystability of $(\Delta, {\bf L}, f)$ in Case~2, as well as in Case~1 when $q(z)= (z-\alpha)^2$, i.e. when $g_+$ is an orthotoric K\"ahler metric, or of parabolic type in the terminology of \cite{ambitoric1}. In general, we expect that there should exist unstable labelled quadrilaterals  $(\Delta, {\bf L}, f)$ with $\mathfrak{F}_{\Delta, {\bf L}, f}=0$,  but we do not show this here. \end{rem}

\subsubsection{Examples} We can use Proposition~\ref{duality} to generate many orbifold examples of Einstein--Maxwell metrics. Indeed, let $\Delta$ be any compact convex quadrilateral in $\R^2$ whose vertices are rational points, and which is not a trapezoid. It is shown in \cite{legendre} that there exists a unique,  up to an overall scale,  labelling ${\bf L}=\{L_1, \ldots, L_4\}$ of $\Delta$ such that $(\Delta, {\bf L})$ admits a positive regular ambitoric K\"ahler--Einstein metric.  It follows by \cite[Thm.~1]{ambitoric2}, that for any small positive rational numbers   $r_i$, the  labelled quadrilateral $(\Delta, {\bf L}_r)$ with ${\bf L}_r = \{\frac{1}{r_1}L_1, \ldots, \frac{1}{r_4}L_4\}$ admits a positive regular ambitoric extremal K\"ahler metric with positive scalar curvature $s$.
The correspondence given by Proposition \ref{duality} together with  Proposition \ref{ambitoric-compactification} then allows us
to associate to $(\Delta, {\bf L}_r)$ another labelled polytope $(\tilde\Delta , \tilde{\bf L}_r)$ which corresponds to the
ambitoric compactification of a positive regular Einstein-–Maxwell metric with conformal factor  $f = 1/s$  and same polynomials
$A(z)$ and $B(z)$ as for the extremal metric, but a different $q$.
This yields, for a given polytope $\Delta$, a family of examples,  parametrized by the four rational numbers $r_1, \ldots, r_4$  (note that an overall scale  of $r_i$'s introduces homothetic metrics). The polytopes $\Delta$ themselves are parametrized, up to an affine equivalence, by two rational parameters. We thus obtain
\begin{cor}\label{c:existence} There exists a continuous-countable family, depending on $5$  rational parameters, of non-homothetic simply-connected, compact conformally--K\"ahler, Einstein--Maxwell toric $4$-orbifolds whose second Betti number is $2$.
\end{cor}

\section{Further observations}\label{varia}
\subsection{Towards a quantized Futaki invariant on a polarized variety and $K$-stability}\label{s:quantized}  In Section~\ref{s:futaki-kahler-class} (see Corollary~\ref{c:complex-futaki}), we defined a Futaki invariant $\mathfrak{F}_{\Omega, K, a} : \mathfrak{g} \to \R$,  associated to a given K\"ahler class $\Omega$ on $(M,J)$, a vector field $K \in \mathfrak{g}$ (where $\mathfrak{g} = {\rm Lie}(G)$ for a compact subgroup $G \subset {\rm Aut}_r(M,J)$),  and a normalization constant $a>0$ which determines a positive Killing potential $f_{K, \omega, a}$ for $K$ for each K\"ahler metric $\omega \in \Omega$.

In the case when $\Omega= 2\pi c_1(L)$,  where $L$ is an ample holomorphic line bundle over  $(M, J)$, one expects to obtain a quantized version of $\mathcal{F}_{\Omega, K,a}$, i.e. an expression  which does not depend on the choice of a reference K\"ahler metric $\omega\in \Omega$,  similarly to \cite{Do-02,tian-old} and the ``twisted modified Futaki invariants'' introduced in \cite{BW, DS, D}.

To this end, we assume that $G \subset {\rm Aut}_r(M,J)$ is  an $\ell$-dimensional real torus which, without loss (by replacing $L$ by a suitable positive tensor power if necessary~\cite{kob-0}), admits a lifted action on the total space of $L$.  We denote by $\Lambda \subset \mathfrak{g}={\rm Lie}(G)$ the lattice of circle  subgroups of $G$,  and by $\Lambda^* \subset \mathfrak{g}^*$ the dual lattice. We consider the induced linear actions of $G$ on the finite dimensional vector spaces $H^0(M, L^{\otimes k})$ of holomorphic sections of $L^{\otimes k}$, and denote by
\begin{equation*}
H^0(M, L^{\otimes k}) = \bigoplus_{\lambda_i^{(k)} \in P_k} E_{\lambda_i^{(k)}}
\end{equation*}
the decomposition of $H^0(M, L^{\otimes k})$ as a direct sum of  common eigenspaces  for the commuting infinitesimal actions of elements in $\mathfrak{g}$; here $\lambda_i^{(k)}\in \mathfrak{g}^*$ are the weights of the complexified action of $G$, or equivalently, for each $H \in \mathfrak{g}$, $\{\lambda_i^{(k)}(H), i=1, \ldots, N_k\}$ is the spectrum, counted with multiplicity, of $(-iA^{(k)}_H)$, where $A^{(k)}_H$ denotes the induced infinitesimal action  of $H$  on $H^0(M, L^{\otimes k})$ and  $N_k = {\rm dim}_{\C}(H^0(M, L^{\otimes k})$. It is well-known that $\lambda_i^{(k)}\in \Lambda^*$. Also, for each $k\in \N$ we let $P_k$
denote the finite set of weights. Following \cite{BW,W},  one defines the {\it spectral  measure} on $\mathfrak{g}^* \cong \R^{\ell}$ by
\begin{equation*}\label{spectral}
\nu_k :=  \sum_{i=1}^{N_k} \delta_{{\lambda_i^{(k)}}/{k}},
\end{equation*}
which is supported on $P_k$.

We now fix a $G$-invariant hermitian metric $h$ on $L$ (which induces a hermitian metric $h^{(k)}$ on each $L^{\otimes k}$), with hermitian connection $\nabla$ and curvature $i\omega$,  where $\omega \in \Omega$ is a K\"ahler form on $(M,J)$. Let $\mu : (M, \omega) \to \mathfrak{g}^*$ be the corresponding moment map and $H_1, \ldots, H_{\ell} \in \mathfrak{g}$ be $\ell$ circle generators which allow us to  identify $\mathfrak{g}\cong \R^{\ell}$ and use the corresponding  momenta $(\mu_1, \ldots, \mu_{\ell})$  in order to regard the moment map as  $\mu : M \to \R^{\ell}$. Let us denote $\lambda_{ij}^{(k)}=\lambda_{i}^{(k)}(H_j) \in \Z^{\ell}$. It is then shown in \cite{W} that for any smooth function $\Phi(x_1, \ldots, x_{\ell})$ on $\R^{\ell}$,
\begin{equation}\label{b0-general}
\begin{split}
\lim_{k\to \infty} \frac{\nu_k(\Phi)}{k^m} &= \lim_{k\to \infty} \frac{1}{k^m}\sum_{i=1}^{N_k}\Phi(\lambda_{i1}^{(k)}/k, \ldots, \lambda_{i\ell}^{(k)}/k)\\
                                                                     & =\frac{1}{(2\pi)^m}\int_M \Phi(\mu_1, \ldots, \mu_{\ell}) v_{\omega}.
\end{split}
\end{equation}
This implies in particular that the expression on the second line is independent of the choice of a K\"ahler metric $\omega\in \Omega$, a result previously known from Futaki--Mabuchi~\cite{FM}. We note here that the above equality holds for a function $\Phi$ which is smooth on $P= \mu(M)$ (i.e. the convex hull in $\mathfrak{g}^*$ of the images of the fixed points of the $G$-action) and on the convex hull  of each $P_k$, see \cite{BW}.

We shall now specialize to the case when  $\mu_1>0$ on $M$,  and $\Phi(x_1, \ldots, x_{\ell}):= x_1^p x_j^q$ with $p, q \in \Z, q\ge 0$. Recall that (see e.g. \cite{gauduchon-book})  for a fixed  $G$-invariant hermitian metric $h$ on $L$ with  curvature $i\omega$, $\omega  \in \Omega$,  the hamiltonian function  $\mu_j$ of $H_j$ with respect to $\omega$ is defined up to  the choice of an integer additive constant (corresponding to different lifts of the action of  $G$ to $L$). Once these constants are  fixed (in our case the only constraint is  that  $\mu_1>0$),  the generators for the lifted $S^1$-actions on $L$  are given by $\hat H_j = H_j^{h} - \mu_j T$,  where $H_j^h$ denotes the horizontal lift of $H_j$ with respect to the Hermitian connection $\nabla$ on $L$, and $T$ is the generator of the natural (diagonal) $S^1$-action on the fibres of $L$. The induced infinitesimal $S^1$-actions on a smooth section $s\in \Gamma(L^{\otimes k})$ are therefore
\begin{equation}\label{induced-action}
A_{H_j}^{(k)} (s)= \nabla_{H_j} s + i k \mu_j s.
\end{equation}
Let  $\lambda_{i1}^{(k)}$ be (a real) eigenvalue for $(-iA_{H_1}^{(k)})_{H^0(M,L^{\otimes k})}$, corresponding to an eigenvector $s_i^{(k)} \in H^0(M,L^{\otimes k})$. By \eqref{induced-action},
\begin{equation*}
\lambda_{i1}^{(k)} |s_i^{(k)}|^2_h = -\frac{i}{2} (d|s_i^{(k)}|_h^2)(H_1) - \frac{1}{2} (d|s_i^{(k)}|_h^2)(JH_1) + k\mu_1 |s_i^{(k)}|^2_h,
\end{equation*}
showing that $\lambda_{i1}^{(k)} >0$  (as we have chosen the lifted action of $G$ so that $\mu_1$  is positive on $M$).  It thus follows that $\Phi= x_1^p x_j^q$ is smooth on $P=\mu(M)$ and on the convex hull of each $P_k$. Let
\begin{equation*}
w^{p,q}_{H_1,H_j}(L^{\otimes k})  :=  \nu_k(x_1^px_j^q)= {\rm trace} \Big(-\frac{i}{k} A_{H_1}^{(k)}\Big)^{p} \Big(-\frac{i}{k}A_{H_j}^{(k)}\Big)^{q}.
\end{equation*}
As noticed in the proof of \cite[Prop.~4.4]{BW},  the Bergman density asymptotic expansion (see e.g. \cite{marinescu-ma})  implies that
\begin{equation}
w^{p,q}_{H_1,H_j}(L^{\otimes k})  \simeq \frac{1}{(2\pi)^m}\sum_{r=0}^{\infty} b_r^{p,q}(H_1,H_j) k^{m -r},
\end{equation}
where, according to  \eqref{b0-general},  the first term is
\begin{equation}\label{b0}
b_0^{p,q}(H_1,H_j) = \int_M \mu_1^p \mu_j^q v_{\omega}.
\end{equation}
For the second term, we expect in general
\begin{equation}\label{b1}
\begin{split}
b_1^{p,q}(H_1,H_j) =& \frac{1}{4}\int_M s_{\omega} \mu_1^p \mu_j^q v_{\omega}  + \frac{pq}{2} \int_M  \mu_1^{p-1} \mu_j^{q-1}g(H_1, H_j) v_{\omega} \\
                         &+ \frac{p(p-1)}{4}\int_M \mu_1^{p-2} \mu_j^{q}g(H_1,H_1)v_{\omega} \\
                         & + \frac{q(q-1)}{4} \int_M \mu_1^{p} \mu_2^{q-2}g(H_j, H_j) v_{\omega},
\end{split}
\end{equation}
where $g(\cdot, \cdot)=\omega(\cdot, J\cdot)$ is the corresponding K\"ahler metric and $s_{\omega}$ is its scalar curvature.

\begin{rem}
To the best of our knowledge,   \eqref{b1} is established only  when $p, q\in\{0,1\}$, see \cite[Prop.~8.9.2]{gauduchon-book}. It is plausible  that the argument extends for any $p, q \in \N$. In the case when $p <0$, one could try to use an approximation argument as in the proof of \cite[Prop. 4.4]{BW} (i.e.  approximate $x_1^{p}$ by  polynomials) in order to establish  \eqref{b1} for any $p$,  but we do not explore this idea in the present paper.
\end{rem}

Let us now consider the situation when the vector fields $K$ and $H$ in the definition of the Futaki invariant $\mathfrak{F}_{\Omega, K, a} (H)$ in Corollary~\ref{c:complex-futaki} are the generators $H_1$ and $H_j$ respectively, with  momenta $\mu_1=f_{K,\omega,a}$ and $\mu_j=h_{H,\omega, b}$.  Assuming that  \eqref{b1} holds true, it then follows from  \eqref{rough-futaki}, \eqref{c-complex} and \eqref{futaki-integral} that
\begin{equation}\label{futaki-algebraic}
\begin{split}
\frac{1}{4}\Big(b_0^{-2m-1,0}(K,H)\Big) \mathfrak{F}_{\Omega, K, a}(H) =&b_1^{-2m+1,1}(K,H)b_0^{-2m-1,0}(K,H) \\
                                                                                     &- b_0^{-2m-1,1}(K,H)b_1^{-2m+1,0}(K,H).
                                                                                     \end{split}\end{equation}

To further motivate the formula \eqref{futaki-algebraic}, we observe that it does hold true in the toric case: switching from the complex to the symplectic point of view, $(M,L)$ is then described in terms of the corresponding  compact  Delzant convex polytope $\Delta \subset \mathfrak{t}^*$ (where $\mathfrak{t}$ is the Lie algebra of the $m$-dimensional torus $\T$) with labels  ${\bf L}=\{ L_j(\mu)= \langle u_j, \mu\rangle + \lambda_j , j=1, \ldots, d\}$ where the $u_j$ belong to the lattice $\Lambda \subset \mathfrak{t}$ of circle subgroups of $\T$. The fact that $M$ is a smooth polarized variety  tells us  that the vertices of $\Delta$ belong to the dual lattice $\Lambda^* \subset \mathfrak{t}^*$, and at each vertex of $\Delta$,  the adjacent normals span $\Lambda$. Taking any generators of $\Lambda$ as a basis of $\mathfrak{t}$  identifies $\Lambda$ and $\Lambda^*$ with $\Z^m$ and introduces a natural euclidean measure $d\mu$ on $\mathfrak{t}^*\cong \R^m$.
A central fact in the theory of algebraic toric varieties is that the weight decomposition of $H^0(X, L^k)$ with respect to the complexified  torus $\T^c$ is isomorphic to
$P_k= \{\mu \in k\Delta \cap\Z^m \}$ with  the weights identified with corresponding elements of $\Z^m$.  On the other hand, for any continuous function $\Phi$ on $\mathfrak{t}^*$, we have for $k$  large enough (see \cite{GS}):
$$\sum_{\lambda \in k\Delta \cap \Z^m} \Phi(\lambda/k) = k^m \int_{\Delta} \Phi d\mu + \frac{k^{m-1}}{2}\int_{\partial \Delta} \Phi d\sigma + O(k^{m-2}).$$
This formula implies that
\begin{equation*}
b_0^{p,q}(H_1,H_j) = (2\pi)^m \int_\Delta \mu_1^p \mu_j^q d\mu, \ \  b_1^{p,q}(H_1, H_j) = \frac{1}{2} (2\pi)^m \int_{\partial \Delta} \mu_1^p \mu_j^q d\sigma.
\end{equation*}
Letting  $\mu_1=f_{K, \omega, a}$ and  $\mu_2=h_{H, \omega, b}$ be the  affine functions on $\tor^*$ corresponding to the vector fields $K$ and $H$, it follows  that \eqref{futaki-algebraic}, suitably manipulated,  is  equal to $(2\pi)^m$ times $\mathfrak{F}_{(\Delta, L, f_{K, \omega, a})} (h_{H, \omega, b})$ (see Sect.~\ref{s:toric} for the definition), which in turn, by the results in Sect.~\ref{s:toric}, equals $\mathfrak{F}_{\Omega, K, a}(H)$.

\smallskip

The quantized version of the Futaki invariant discussed above  would lead to a general notion of $K$-polystability  (similar to \cite{Do-02} in the case $K=0$) which will be relevant to the existence of Einstein--Maxwell metrics in the conformal classes of K\"ahler metrics in $c_1(L)$.  In the toric case, using the  toric degenerations and the computation from \cite{Do-02}, one can
obtain the conclusion that the $K$-polystability of $(\Delta, {\bf L}, f_{K, \omega, a}))$ over {\it rational} PL convex functions is a necessary condition for the existence of an Einstein--Maxwell metric in $c_1(L)$ (associated to a choice of $K$ and $a$). Note that one might also define $K$-polystability for special degenerations (in the terminology of Tian~\cite{tian}), which are also $G$-equivariant. It will be interesting to see even in this special case  whether the existence of a conformally K\"ahler,  Einstein--Maxwell metric associated to the triple $(\Omega, K, a)$ would imply that on the central fibre $(M_0, J_0, \Omega_0)$, the Futaki invariant $\mathcal{F}_{\Omega_0, K,a}(H_0)\ge 0$,  where $H_0$ is the vector field corresponding to the circle action of the induced $\C^{\times}$-action on $M_0$. We will not pursue these matters here.

\subsection{Characterization of the conformally-K\"ahler, Einstein metrics}\label{s:conformally-einstein} It is well-known since the work of Derzinski~\cite{De} that there exists an Einstein metric in the conformal class of K\"ahler metric on a $4$-dimensional manifold if and only if the (conformal) Bach tensor vanishes.  For the examples  discussed in Sect.~\ref{s:examples}, the vanishing of the Bach tensor is expressed in \cite{ACG,ambitoric1,ambitoric2, LeB} in terms of an additional quadratic relation between the coefficients of the polynomials $A(z)$ and $B(z)$. It would be natural to try to extend such a characterization in general (and possibly  to higher dimensions too). As a first step in this direction, we notice the following
\begin{lemma}\label{Einstein-local} Let $\tilde g = \frac{1}{f^2}g$ be a conformally-K\"ahler, Einstein--Maxwell metric on a connected $4$-dimensional manifold $(M,J)$. Then the following conditions are equivalent
\begin{enumerate}
\item[\rm (i)] $g$ is an extremal K\"ahler metric and $f$ is a multiple of the scalar curvature $s_g$ of $g$;
\item[\rm (ii)] $\tilde g$ is either Einstein, or is a cscK metric homothetic to $g$.
\end{enumerate}
\end{lemma}
\begin{proof}
${\rm (ii)\Rightarrow (i)}$ This is established in \cite{De}.

${\rm (i)\Rightarrow (ii)}$  Let $g$ be an extremal K\"ahler metric with non-vanishing scalar curvature $s_g$. Note that, as $s_g$ is a Killing potential by definition, $s_g$ is either constant, or $ds_g \neq 0$ on a dense open subset of $M$.  We are going to show that in the latter case $\tilde g$ must be Einstein.

The metric $\tilde g = \frac{1}{s_g^2} g$ satisfies $\delta^{\tilde g} W^+=0$, where $W^+$ is the (conformally invariant) selfdual Weyl tensor of $g$ (with respect to the canonical orientation of $(M,J)$). The Bach tensor of $\tilde g$ is then given by \cite{De} (see also \cite{ACG}):
$$B^{\tilde g}_{X,Y} = \frac{1}{2}\sum_{i=1}^4{\tilde g}(W^+_{e_i, X} r^{\tilde g}_0(e_i), Y).$$
Using the fact that the traceless Ricci endomorphism $r^{\tilde g}_0$ is $J$-invariant and the selfdual Weyl tensor of the K\"ahler surface $(g, J)$ is (see e.g. \cite{De})
$$g(W^+_{X,Y} Z, T)= \frac{s_g}{12} (\omega \otimes \omega)_0(X,Y,Z,T),$$
where $(\omega \otimes \omega)_0$ stands for the trace-free part of $\omega \otimes \omega$ viewed as an endomorphism of the bundle of selfdual 2-forms $\Lambda^+(M)$ via the metric $g$, one gets
$$B^{\tilde g}_{X,Y} = \frac{s_g^3}{12}{\rm Ric}_0^{\tilde g}(X,Y).$$
Since the Bach tensor is a gradient of the Weyl functional (see \cite{besse}), one has $\delta^{\tilde g} B^{\tilde g}=0$ whereas $\delta^{\tilde g} {\rm Ric}_0^{\tilde g}=0$ by the Ricci identity (recall that $\tilde g$ has constant scalar curvature by the Einstein--Maxwell assumption). It thus follows that
\begin{equation}\label{local}
{r}_{0}^{\tilde g}(d s_g^{\sharp})=0.
\end{equation}
Since ${r}_{0}^{\tilde g}$ is $J$-invariant and trace-free,  it has at any given point real eigenvalues $(\lambda, \lambda, -\lambda, -\lambda)$, showing that $(r^{\tilde g}_0)^2 = \frac{|{\rm Ric}^{\tilde g}_0|^2_{\tilde g}}{4}{\rm Id}$. Applying $r_0^{\tilde g}$ to \eqref{local}, we conclude $|{\rm Ric}^{\tilde g}_0|^2_{\tilde g} ds_g \equiv 0$.  This shows that ${\rm Ric}^{\tilde g}_0$ vanishes on a dense open subset, therefore everywhere on $M$.
\end{proof}

\subsection{Computing the Futaki invariant}\label{s:computing-futaki}

\subsubsection{$f$-extremal metrics and $f$-extremal vector field}\label{s:extremal} It is a well-known fact~\cite{FM} that the vanishing of the usual Futaki invariant associated to a K\"ahler class  $\Omega$ on $(M,J)$ is  equivalent to the vanishing of the so-called {\it extremal vector field} $Z_{\Omega, G}$ associated to some (and hence any) maximal torus $G\subset {\rm Aut}_r(M,J)$. One can easily extend this notion to the setting of  Sect.~\ref{s:futaki}.

Consider first the symplectic point of view (Sect.~\ref{s:symplectic-futaki}), where the symplectic form $\omega$,  the torus $G\subset {\rm Ham}(M,\omega)$,  and the positive hamiltonian function $f$  are fixed. It follows from Corollary~\ref{c:symplectic-futaki} that the projection $z_{G,\omega,f}$ of $s_{J,f}$ with respect to $\langle \cdot, \cdot \rangle_f$ to the finite dimensional vector space of hamiltonians for the elements of $\mathfrak{g}= {\rm Lie}(G)$ is independent of the choice of $J\in \mathcal{C}_{\omega}^G$.  We denote by $Z_{G,\omega,f} = {\rm grad}_{\omega} (z_{G, \omega, f}) \in \mathfrak{g}$ the corresponding vector field on $M$.
\begin{defn} The  function $z_{G,\omega,f}$ and the hamiltonian vector field $Z_{G,\omega,f}$ it determines  will be referred to as {\it $f$-extremal potential} and {\it $f$-extremal vector field}, respectively. A K\"ahler metric $g_J$ corresponding to  $J\in \mathcal{C}^G_{\omega}$ will be called {\it $f$-extremal} if the scalar curvature $s_{J,f}$ of $\frac{1}{f^2}g_J$ is a hamiltonian for an element in $\mathfrak{g}$, or equivalently, if  $s_{J,f}=z_{G, \omega, f}$.
\end{defn}
It is then clear from Corollary~\ref{c:symplectic-futaki} that the Futaki invariant $\mathfrak{F}^G_{\omega, f}$ vanishes if and only if $Z_{G,\omega, f}=0$~\footnote{If $\mathfrak{g}$ is identified with the space of  corresponding hamiltonians $h$ satisfying $\langle h, 1 \rangle_f =0$,  then $\langle \cdot, \cdot \rangle_f$ induces an euclidean product on $\mathfrak{g}$ with respect to which  $Z_{G, \omega, f}\in \mathfrak{g}$ and $\mathfrak{F}^G_{\omega, f} \in \mathfrak{g}^*$ are dual.}, in which case $z_{G,\omega, f} = c_{\omega, f}$. Theorem~\ref{thm:moment-map-setting}  and Remark~\ref{AK-futaki} (or, equivalently, formula \eqref{long-computation} with $h=s_{J,f}$) imply that $g_J$ is $f$-extremal if and only if it is a critical point of the functional $\int_M \frac{s_{J,f}^2}{f^{2m+1}} v_{\omega}$ over the space $\mathcal{AK}_{\omega}^G$, see also \cite{lejmi}.

The above considerations have a natural counterpart in the setting of Sect.~\ref{s:futaki-kahler-class} where the complex structure $J$, the torus $G\subset {\rm Aut}_r(M,J)$, a vector field $K\in \mathfrak{g}$ and a normalization constant $a>0$ are fixed, and $\omega$ varies within the space $\mathcal{K}_{\Omega}^G$ of $G$-invariant K\"ahler metrics in a given K\"ahler class $\Omega$. To this end, a simple application of the $G$-equivariant Moser lemma (compare with \cite[Rem.~3.3]{lejmi}) shows that the $f_{\omega, K, a}$-extremal vector field $Z_{G, \omega, f_{K, \omega, a}}$ of $(M,\omega)$ does not depend on the choice of $\omega \in \mathcal{K}_{\Omega}^G$, and therefore gives rise to a well-defined real holomorphic vector field $Z_{\Omega, G, K, a}$ associated to $(\Omega, G, K, a)$. By the previous discussion,   $Z_{\Omega, G, K, a}$ vanishes if and only if the corresponding Futaki invariant $\mathfrak{F}^G_{\Omega, K, a}$ is zero.

\subsubsection{The Futaki invariant of a weighted projective plane}

To illustrate the notions introduced in Sect.~\ref{s:extremal}, let us consider the case when $(M, J)$  is a weighted projective plane $\mathbb{C}P^2_{a_0,a_1,a_2}$ (see Sect.~\ref{s:wpp}) with $G=\T$ being the (maximal) $2$-dimensional torus  corresponding to the diagonal action of $(\C^{\times})^3$ in homogeneous coordinates. As noticed in Remark~\ref{r:wpp}, up to a homothety of the symplectic form, one can alternatively describe $\mathbb{C}P^2_{a_0,a_1,a_2}$ as a toric symplectic orbifold $(M,\omega)$ corresponding to the labelled standard simplex
$(\Delta, {\bf L}^{\bf a})=\{L_0^{\bf a}= a_1a_2(1-\mu_1-\mu_2)>0,  L_1^{\bf a}=a_0a_2 \mu_1 >0,  L_2^{\bf a}=a_0a_1 \mu_2 >0\}$. To account for the homothety factors of the symplectic form, we shall consider  more generally  labellings of the form ${\bf L}^{\bf r}=\{L_0^{\bf r}=\frac{1}{r_0}(1-\mu_1 -\mu_2), L^{\bf r}_1=\frac{1}{r_1} \mu_1, L^{\bf r}_2=\frac{1}{r_2}\mu_2\}$ with $r_i = a_i/\lambda$, where $\lambda >0$ is a real number. We notice that for any choice of positive rational numbers $r_0\ge r_1\ge r_2>0$, $(\Delta, {\bf L}^{\bf r})$ parametrizes a weighted projective plane for some weights ${\bf a}=(a_0,a_1, a_2)$. We then have
\begin{prop}\label{wpp-futaki} Let $g_{BF}$ be the Bochner-flat K\"ahler metric on $\C P^2_{a_0, a_1, a_2}$ with $a_0\ge a_1\ge a_2 >0$, viewed as an $\omega$-compatible toric K\"ahler metric associated to the standard simplex $(\Delta, {\bf L}^{\bf a})$ and $f$ be (the pull-back by the momentum map of) a positive affine-linear function on $\Delta$. Then,
$\tilde g_{BF,f}=\frac{1}{f^2}g_{BF}$ is $f$-extremal. In particular, the Futaki invariant $\mathfrak{F}^{\T}_{\omega, f}$ vanishes if and and only if $\tilde g_{BF,f}$ is Einstein--Maxwell, i.e. if and only if  $a_0 < a_1 + a_2$ and $f= \lambda s_{BF}$, where $s_{BF}$ is the scalar curvature of $g_{BF}$.
\end{prop}
\begin{proof} We first prove that $\tilde g_{BF}$ is $f$-extremal, or equivalently, that that the scalar curvature of $\tilde g_{BF,f}= \frac{1}{f^2} g_{BF}$ is a Killing potential. By using the explicit expression~\cite{abreu,Bryant}
$$u_{BF} = \frac{1}{2} \Big[\sum_{j=0}^{2}L_j^{\bf r}  \log L_j^{\bf r}  +  \Big(\sum_{j=0}^2L_{j}^{\bf r}\Big)\log\Big(\sum_{j=0}^2 L_{j}^{\bf r}\Big)\Big],$$
for the symplectic potential $u_{BF}$ of $g_{FB}$ and \eqref{sg},  our claim is equivalent to check  that
\begin{equation}\label{key-observation}
f^5 \Big[\sum_{i,j=1}^2 \Big( \frac{1}{f^3}({\bf H}^{u_{BF}})_{ij}\Big)_{,ij}\Big]
\end{equation}
is an affine function in $(\mu_1, \mu_2)$, where we recall ${\bf H}^{u_{BF}}= \big({\rm Hess}(u_{BF})\big)^{-1}$.

Alternatively,  as \eqref{key-observation} manifestly depends continuously on the labels ${\bf L}^{\bf r}$ (i.e. on the rational parameters ${\bf r}=(r_0, r_1, r_2)$), it is enough to establish that \eqref{key-observation} is affine for positive rational triples ${\bf r}=(r_0,r_1,r_2)$ with $r_i \neq r_j$ for $i\neq j$. In this case, the Bochner-flat metric is {\it orthotoric} \cite{ACGT2,Bryant},  i.e. $g_{BF}=g_+$ where $g_+$ is the K\"ahler metric given by \eqref{g-pm}, with $q(z)\equiv 1$, $A(z)=R(z)= -B(z)$,  where $R(z)$ is a polynomial of degree $4$ with $4$ distinct real roots. Note that in the notation of Sect.~\ref{s:ambitoric},  we have (using $q(z)=1$) $\mu_1= x+y, \mu_2= xy$ so that the affine function $f$ can be written as $f(x,y)= p_2xy + p_1(x+y) + p_0$ where $p(z)=p_2z^2 + 2p_1z + p_0$ is some element of $S^2$.  Now, as $A(z)$ and $B(z)$ satisfy the first two conditions of \eqref{regular-EM} with $\rho(z)=0$, by Remark~\ref{r:ambitoric-scalar} the scalar curvature of $\tilde g=\frac{1}{f^2}g_{+}$ is $s_{\tilde g} =-\{p, (p, R)^{(2)}\}(x,y)$, which is the polarization of a polynomial of degree $\le 2$, i.e. an affine function in momenta $(\mu_1, \mu_2)$.

According to the discussion in Sect.~\ref{s:extremal}, the scalar curvature of $\tilde g_{BF,f}=\frac{1}{f^2} g_{BF}$ is the $f$-extremal function of $(\Delta, {\bf L}^{\bf r},f)$. It is constant if and only if $\mathfrak{F}^{\T}_{\omega, f} = (2\pi)^2\mathfrak{F}_{(\Delta, {\bf L}^{\bf r}, f)}=0$, which holds if and only if  $\tilde g_{BF,f}$ is Einstein--Maxwell. This and Theorem~\ref{wpp-classification} conclude the proof. \end{proof}

\begin{rem} The arguments in the proof of Proposition~\ref{wpp-futaki} reveal an interesting  algebraic structure of the subset of the positive affine functions $f$ for which $\mathfrak{F}^{\T}_{\omega, f}=0$. Indeed, when the weights are
mutually  distinct, the latter subset is identified with the subspace $F \subset \Proj(S^2) \cong \R P^2$ of projective classes of polynomials $p(z)=p_2z^2 + 2p_1z + p_0$ for which $\{p, (p, R)^{(2)}\}$ is of degree $0$. Given a polynomial  $R(z)$ of degree $\le 4$, this realizes  $F$ as the intersection of two quadrics in $\R P^2$ (given by the coefficients before the $z^2$ and $z$ of $\{p, (p, R)^{(2)}\}(z)$). They always intersect at the point $[p(z)]=[R'''(z)]$ (which geometrically corresponds to taking $f= \lambda s_{BF}$), but they may also have other common points. For instance,  it is not difficult to check that  if the coefficients of $R(z)=\sum_{i=0}^4 a_i z^i$ satisfy $a_3^2a_0 = a_1^2a_4$,  then $[p(z)]= [-a_3z^2 + a_1]$ provides  a different solution. (One can show that this is essentially the only other point in the intersection.) However, according to Proposition~\ref{wpp-futaki}, in the latter case $R(z)$ cannot have $4$ real roots $\alpha_0<\alpha_1<\alpha_2<\alpha_3$ while at the same time $p(x,y)=-a_3xy + a_1$ be positive on the product of intervals $[\alpha_0,\alpha_1]\times [\alpha_1, \alpha_2]$. \end{rem}

\subsubsection{The Futaki invariant of $\C P^1 \times \C P^1$ and  the classification of toric Einstein--Maxwell metrics}

Let $M=\mathbb{C}P^1\times \mathbb{C}P^1$. Aside from the obvious constant scalar curvature K\"ahler product metrics
in each K\"ahler class, LeBrun~\cite{LeB0} gave an explicit construction of conformally K\"ahler, non-K\"ahler Einstein--Maxwell metrics on $(M,J)$. In terms of  Theorem~\ref{ambitoric-EM-classification}, the LeBrun metrics are ambitoric compactifications of the product Einstein--Maxwell ansatz  over a labelled parallelogram $(\Delta, {\bf L})$, with the positive affine function $f$ being such that the  affine line $f=0$ it determines is parallel to a pair of parallel facets of $\Delta$. On the other hand, we observed in  Remark~\ref{extension} that in this case the vanishing of the Futaki invariant $\mathfrak{F}_{\Delta, {\bf L}, f}$ is necessary and sufficient condition for the existence of conformally K\"ahler, Einstein--Maxwell metric with conformal factor $f$, so we are in position to use the theory from Sect.~\ref{s:toric} to give here a computer-assisted proof that the non-K\"ahler, conformally--K\"ahler,  Einstein--Maxwell metrics found in \cite{LeB0} are essentially the unique ones on $\C P^1 \times \C P^1$  for which  the  isometry group contains a $2$-dimensional torus.

\begin{prop}\label{thm:classification-product} Suppose $\tilde g = \frac{1}{f^2} g$ is a conformally K\"ahler,  Einstein--Maxwell metric on $(M,J)= \C P^1\times \C P^1$ which is not cscK, and  whose isometry group contains a  $2$-dimensional real torus. Then $\tilde g$ must be homothetically isometric  to one of the LeBrun Einstein--Maxwell metrics constructed in \cite{LeB0}. \end{prop}

\begin{proof} By assumption,  $\tilde g$ is invariant under the action of a $2$-dimensional torus $\T$ in the connected component of identity in its isometry group. This action must preserve the complex structure $J$ as otherwise $\tilde g$ would admit a continuous family of positively oriented compatible complex structures (the pull-backs of $J$ by the isometric action of $\T$) and,  therefore,  $\tilde g$ must be anti-self-dual (see e.g. \cite{pontecorvo}). This is impossible, as the conformal K\"ahler metric $g$ must then be of zero scalar curvature (see e.g.~\cite{D}), i.e. $c_1(M,J)\cdot [\omega]=0$,  contradicting the fact that $\C P^1 \times \C P^1$ is a Fano complex surface. Furthermore, as $\C P^1 \times \C P^1$ is simply connected, ${\rm Aut}(M,J)= {\rm Aut}_r(M,J)$. Thus the metric $\tilde g$ is toric, in the sense given in the beginning of Sect.~\ref{s:toric}.

Fix a symplectic form $\omega_{S^2}$ of volume $4\pi$ on $S^2 = \C P^1$,  and an $S^1$ hamiltonian action with momentum image the interval $[-1, 1]$. By rescaling, we can assume $g$ is compatible to the symplectic form $\omega_\lambda= \omega_{S^2} \oplus \lambda \omega_{S^2}$, for some  $\lambda>0$.  The Delzant image of $(M,\omega_\lambda)$ is the Delzant rectangle  $\Delta^{\lambda} = [-1,1]\times [-\lambda, \lambda]$  with the standard labelling ${\bf L}=\{1-x , 1+ x, \lambda-y, \lambda + y \}$. Using a homothety of the $y$-variable, we can send $(\Delta^{\lambda}, {\bf L})$ to the labelled square $(\Delta^{1}, {\bf L}^{\lambda})$ where $\Delta=\Delta^{1}=[-1,1]\times [-1,1]$ and ${\bf L}^{\lambda}=\{1-x, 1+x, {\lambda}(1-y), {\lambda}(1+y)\}$, a normalization we are going to use. As $f$ is a $\T$-invariant Killing potential for $(g, \omega_\lambda)$, $f(x,y)= f_0 + f_1x + f_2 y$. It follows from Theorem~\ref{toric-futaki} that the Futaki invariant $\mathfrak{F}_{\Delta, {\bf L}^{\lambda}, f}=0$. In order to compute the Futaki invariant as a function of the coefficients $(f_1, f_2)$ of $f$, we are going to use formulae \eqref{c0} and \eqref{F} with respect to the $\omega_\lambda$-compatible toric product metric
 $g_{\lambda}=g_1\oplus g_2^{\lambda}$, where
$$g_1=\frac {dx^2}{A(x)}+A(x)dt_1^2, \quad g_2^{\lambda}=\frac {dy^2}{B(y)}+B(y)dt_2^2$$
with  $A(x)=1-x^2,\quad B(y)=\frac{1}{\lambda}(1-y^2).$
It is easily seen that $g_{\lambda}$ belongs  to $\mathcal{C}^{\T}_{\omega_\lambda}$. Note that the positive constant  $\mu:=\frac{1}{\lambda}$ is the Gauss curvature of $g_2$.

The scalar curvature of  $\tilde g_\lambda=g_\lambda/f^2$ is  (see \eqref{sg} and \eqref{solution-parallelogram})
\begin{equation*}
\begin{split}
s_{\tilde g_\lambda} = & -f^5\Big[\Big(\frac{A(x)}{f^3}\Big)_{xx}+\Big(\frac{B(y)}{f^3}\Big)_{yy}\Big]\\
                  = &(2(\mu+1){f_{1}}^{2} {x}^{2}+ (2(\mu+1){f_{2}}^{2}){y}^{2}  -(8(\mu+1)f_{1}f_{2}) xy\\
                     & +(4(\mu-2)f_{0}f_{1})x+ (4(\mu-2)f_{0}f_{2}) y+2(\mu+1)f_{0}^{2}\\
                     & -12f_{1}^{2}-12\mu f_{2}^{2}.
                     \end{split}
                     \end{equation*}
We thus have
\begin{equation*}
\begin{split}
c_{\omega_\lambda,f} &=\int_M (s_{\tilde g_\lambda}/f^5)v_{\omega_\lambda}\left/ \int_M (1/f^5)v_{\omega_\lambda}\right. \\
 &=\int_{-1}^1\int_{-1}^1 (s_{\tilde g_\lambda}/f(x,y)^5)dx dy \left/ \int_{-1}^1\int_{-1}^1 (1/f(x,y)^5) dx dy\right..
\end{split}
\end{equation*}
Feeding the explicit expressions to Maple software and performing the iterated integrals, we get
\begin{multline*}
c_{\omega_\lambda,f}=
6\,  \left( 2\,{f_{{0}}}^{2}{f_{{1}}}^{2}-3\,{f_{{1}}}^{4}+2\,{
f_{{1}}}^{2}{f_{{2}}}^{2}+2\mu {f_{{2}}}^{2}{f_{{0}}}^{2}-3\mu{f_{{2}}}
^{4}+\mu{f_{{0}}}^{4}\right.\\
\left.-2\,{f_{{0}}}^{2}{f_{{2}}}^{2}+{f_{{2}}}^{4}+{f_{{0
}}}^{4}+2\mu{f_{{2}}}^{2}{f_{{1}}}^{2}-2\mu{f_{{0}}}^{2}{f_{{1}}}^{2}+
\mu{f_{{1}}}^{4} \right)  \left( f_{{0}}-f_{{1}}-f_{{2}} \right)\\
 \left( f_{{0}}-f_{{2}}+f_{{1}} \right)  \left( f_{{0}}-f_{{1}}+f_{{2}
} \right)  \left( f_{{0}}+f_{{2}}+f_{{1}} \right)
\left/\left(-3\,{f_{{0}}}^{4}{
f_{{1}}}^{2}+3\,{f_{{1}}}^{6}-3\,{f_{{2}}}^{4}{f_{{1}}}^{2}\right.\right.\\
\left.\left.+14\,{f_{{0
}}}^{2}{f_{{2}}}^{2}{f_{{1}}}^{2}-3\,{f_{{0}}}^{2}{f_{{1}}}^{4}-3\,{f_
{{2}}}^{2}{f_{{1}}}^{4}+3\,{f_{{0}}}^{6}-3\,{f_{{0}}}^{4}{f_{{2}}}^{2}
-3\,{f_{{2}}}^{4}{f_{{0}}}^{2}+3\,{f_{{2}}}^{6}\right)\right.
\end{multline*}
We stress that the only explicit assumptions given to the software in
the calculation of these integrals are to prevent it from interpreting
them as improper.

Similarly, we calculate the Futaki invariant evaluated on the affine functions  $h=x$ and $h=y$,
using
$$
\mathfrak{F}_{\Delta, {\bf L}^{\lambda}, f} (h)=
\frac{1}{(2\pi)^2}\int_M \Big( \frac{s_{\tilde g_\lambda}-c_{\omega_\lambda, f}}{f^5}\Big)  h\ v_{\omega_\lambda}
=\int_{-1}^1\int_{-1}^1  \Big( \frac{s_{\tilde g_\lambda}-c_{\omega_\lambda,f}}{f^5}\Big)\,h\,dx\,dy$$
The answers given by the software are
\begin{multline*}
\mathfrak{F}_{\Delta, {\bf L}^{\lambda}, f} (x)=
16\, f_{{1}} \Big( 2\mu{f_{{0}}}^{4}{f_{{2}}}^{2}+5\mu{f_{{0}}}
^{2}{f_{{2}}}^{4}-2\,{f_{{0}}}^{2}{f_{{1}}}^{4}+{f_{{2}}}^{6}-\mu{f_{{2
}}}^{4}{f_{{1}}}^{2}\\
-2\,{f_{{0}}}^{2}{f_{{2}}}^{2}{f_{{1}}}^{2}-4\mu{f_
{{0}}}^{2}{f_{{2}}}^{2}{f_{{1}}}^{2}-3\mu{f_{{0}}}^{4}{f_{{1}}}^{2}+2
\mu{f_{{2}}}^{2}{f_{{1}}}^{4}-\mu{f_{{1}}}^{6}-2\,{f_{{0}}}^{6}+{f_{{0}}
}^{4}{f_{{2}}}^{2}\\
\left.+\mu{f_{{0}}}^{6}+{f_{{2}}}^{2}{f_{{1}}}^{4}-2\,{f_{{2
}}}^{4}{f_{{1}}}^{2}+4\,{f_{{0}}}^{4}{f_{{1}}}^{2}+3\mu{f_{{0}}}^{2}{f
_{{1}}}^{4} \Big)
\left/\Big( \left( f_{{0}}-f_{{2}}+f_{{1}} \right)\right.\right.\\
\left(
f_{{0}}-f_{{1}}-f_{{2}} \right)  \left( -3\,{f_{{0}}}^{4}{f_{{1}}}^{2}
+3\,{f_{{1}}}^{6}-3\,{f_{{2}}}^{4}{f_{{1}}}^{2}+14\,{f_{{0}}}^{2}{f_{{
2}}}^{2}{f_{{1}}}^{2}-3\,{f_{{0}}}^{2}{f_{{1}}}^{4}\right.\\
\left.-3\,{f_{{2}}}^{2}{f
_{{1}}}^{4}+3\,{f_{{0}}}^{6}-3\,{f_{{0}}}^{4}{f_{{2}}}^{2}-3\,{f_{{2}}
}^{4}{f_{{0}}}^{2}+3\,{f_{{2}}}^{6} \right)  \left( f_{{0}}+f_{{2}}+f_
{{1}} \right)  \left( f_{{0}}-f_{{1}}+f_{{2}} \right) \Big),
\end{multline*}
and
\begin{multline*}
\mathfrak{F}_{\Delta, {\bf L}^{\lambda}, f} (y)=
-16\, f_{{2}} \Big(-\mu{f_{{1}}}^{6}+2\mu{f_{{0}}}^{2}{f_{{2}}}^
{2}{f_{{1}}}^{2}-\mu{f_{{0}}}^{4}{f_{{1}}}^{2}+{f_{{2}}}^{6}+{f_{{2}}}^
{2}{f_{{1}}}^{4}\\
-2\,{f_{{2}}}^{4}{f_{{1}}}^{2}+3\,{f_{{0}}}^{4}{f_{{2}
}}^{2}-3\,{f_{{2}}}^{4}{f_{{0}}}^{2}-{f_{{0}}}^{6}+2\mu{f_{{0}}}^{2}{f
_{{2}}}^{4}+2\mu{f_{{0}}}^{6}+4\,{f_{{0}}}^{2}{f_{{2}}}^{2}{f_{{1}}}^{
2}\\
\left.-5\,{f_{{0}}}^{2}{f_{{1}}}^{4}-2\,{f_{{0}}}^{4}{f_{{1}}}^{2}+2\mu{f_
{{2}}}^{2}{f_{{1}}}^{4}-\mu{f_{{2}}}^{4}{f_{{1}}}^{2}-4\mu{f_{{0}}}^{4}{
f_{{2}}}^{2} \Big)
\left/ \Big( \left( f_{{0}}-f_{{2}}+f_{{1}} \right)\right.\right.\\
 \left( f_{{0}}-f_{{1}}-f_{{2}} \right)  \left( -3\,{f_{{0}}}^{4}{f_{{
1}}}^{2}+3\,{f_{{1}}}^{6}-3\,{f_{{2}}}^{4}{f_{{1}}}^{2}+14\,{f_{{0}}}^
{2}{f_{{2}}}^{2}{f_{{1}}}^{2}-3\,{f_{{0}}}^{2}{f_{{1}}}^{4}\right.\\
\left.-3\,{f_{{2}
}}^{2}{f_{{1}}}^{4}+3\,{f_{{0}}}^{6}-3\,{f_{{0}}}^{4}{f_{{2}}}^{2}-3\,
{f_{{2}}}^{4}{f_{{0}}}^{2}+3\,{f_{{2}}}^{6} \right)  \left( f_{{0}}+f_
{{2}}+f_{{1}} \right)  \left( f_{{0}}-f_{{1}}+f_{{2}} \right)\Big).
\end{multline*}
We consider first the case where $f_0$, $f_1$, $f_2$ are all nonzero.
In this case $\mathfrak{F}_{\Delta, {\bf L}^{\lambda}, f}$ vanishes iff
the two homogeneous polynomials in the $f_i$'s appearing in brackets in
the numerators of the last two expressions vanish. Now both
these polynomials are affine in $\mu$, so  isolating $\mu$ from each, the
resulting rational functions must equal each other, and we check whether
this can be the case. Factoring the difference of these two rational
functions, using the software, gives another rational function whose
numerator contains the factors $f_0^2$, the four linear terms
$(f_0\pm f_1\pm f_2)$ and
\begin{multline}\label{six}
-3\,{f_{{0}}}^{4}{f_{{1}}}^{2}+3\,{f_{{1}}}^{6}-3\,{f_{{2}}}^{4}{f_{{1
}}}^{2}+14\,{f_{{0}}}^{2}{f_{{2}}}^{2}{f_{{1}}}^{2}-3\,{f_{{0}}}^{2}{f
_{{1}}}^{4}-3\,{f_{{2}}}^{2}{f_{{1}}}^{4}\\
+3\,{f_{{0}}}^{6}-3\,{f_{{0}}}^{4}{f_{{2}}}^{2}
-3\,{f_{{2}}}^{4}{f_{{0}}}^{2}+3\,{f_{{2}}}^{6}.
\end{multline}
The positivity of $f$ over the vertices of $\Delta$ implies $f_0>0$ and $(f_0\pm f_1\pm f_2)>0$; up to a homothety of $\tilde g_{\lambda}$,  we can also assume $f_0=1$, and therefore $|f_1|<1$ and $|f_2|<1$ again by using the positivity of $f$ over (the vertices of) $\Delta$. Thus,  we obtain that \eqref{six}  must vanish for $f_0=1$ and
$f_1^2$, $f_2^2$  in $(0,1)$. The expression \eqref{six} (with $f_0=1$) becomes a cubic polynomial in $\xi=f_1^2$ and $\eta=f_2^2$:
$$Q(\xi, \eta):=3\eta^3-3(\xi+1)\eta^2-(3\xi^2-14\xi+3)\eta+3(\xi+1)(\xi-1)^2.$$
Now $Q(\xi, 0)>0$ for any $\xi \in (0,1)$, so that the polynomial $P_{\xi}(\eta)=Q(\xi, \eta)$ has one negative real root $\eta$
for any such $\xi$. On the other hand, the discriminant of $P_{\xi}(\eta)$ is
$$-768\xi\Big(9\xi^4-30\xi^3+47\xi^2-30\xi+9\Big),$$ which is negative for $\xi$
in $(0,1)$, as the expression in brackets has a positive minimum in $[0,1]$ at $\xi \approx 0.533$. Thus $P_{\xi}$ has no zeros in $(0,1)$ for any $\xi$ in $(0,1)$, so that
\eqref{six} cannot vanish for $f_1^2$, $f_2^2$ in $(0,1)$.

This shows that either $f_1$ or $f_2$ must be zero, should $\mathfrak{F}_{\Delta, {\bf L}^{\lambda}, f} =0$ for some $f$. Suppose for instance $f_1=0$ (the case $f_2=0$ is similar). Note that as we assumed $\tilde g$ is not cscK, we have in this case $f_2 \neq 0$, and we still normalize $f$ such that $f_0=1$. Thus,  $f$ is positive over $\Delta$ iff $f_2^2 \in (0,1)$.  By the above general expressions for the Futaki invariant,  we have $\mathfrak{F}_{\Delta, {\bf L}^{\lambda}, f}(x)=0$ whereas $\mathfrak{F}_{\Delta, {\bf L}^{\lambda}, f}(y)=0$ reduces to
\begin{equation*}
(f_2^2-1)^2(f_2^2 -1 + 2\mu)=0,
\end{equation*}
i.e. the Futaki invariant $\mathfrak{F}_{\Delta, {\bf L}^{\lambda}, f}$ vanishes if and only if
$$1-2\mu = 1- \frac{2}{\lambda} =f_2^2,$$
showing that $\lambda \in (2, \infty)$  is a necessary and sufficient condition for the existence of such $f$. Furthermore, if
such $f$ exists there are, up to scale, precisely two solutions  of $\mathfrak{F}_{\Delta, {\bf L}^{\lambda}, f}=0$, given by $f^{\lambda}_{\pm}(x,y)= \pm \sqrt{\Big(1- \frac{2}{\lambda}\Big)} y + 1$.

The case $f_2=0$ is treated similarly: one then gets that $\mathfrak{F}_{\Delta, {\bf L}^{\lambda}, f}=0$  if and only if
$$1-2\lambda=f_1^2,$$
showing that $\lambda \in (0, \frac{1}{2})$,  and that in this case  up to scale there are precisely two positive affine functions $f$ with $\mathfrak{F}_{\Delta, {\bf L}^{\lambda}, f}=0$,  given by $f_{\pm}^\lambda(x,y)= \pm \sqrt{(1-2\lambda)} x + 1$.

The conclusion of the above analysis is that $\mathfrak{F}_{\Delta, {\bf L}^{\lambda}, f}$ vanishes for some $f$ if and only if $\lambda \in (0, \frac{1}{2}) \cup (2, \infty)$ and for each such $\lambda$, up to a scaling factor,  there are precisely two values $f=f^{\lambda}_{\pm}$ for which $\mathfrak{F}_{\Delta, {\bf L}^{\lambda}, f}$=0. Furthermore, as the affine line $f=0$ for those values is parallel to a pair of parallel facets of $\Delta$, Remark~\ref{extension} (which relies on the proof of Theorem~\ref{ambitoric-EM-classification} (see Case 3)) and Proposition~\ref{ambitoric-compactification} imply that there
exists an ($\omega_\lambda$-compatible) product ambitoric compactification $g_f$ such that $(\frac{1}{f^2})g_f$ is Einstein-Maxwell, as
was explicitly shown in \cite{LeB0}. The uniqueness Theorem~\ref{uniqueness} then shows that $g$ must be $g_f$, up to a $\T$-equivariant isometry.
\end{proof}

\begin{rem}\label{r:product-non-uniqueness} We will now show that the two solutions of \eqref{modified-abreu} associated to the affine functions $f^{\lambda}_{\pm}(x, y)$  are in fact conformal. Let us consider the case when $\lambda >2$, i.e. $f^{\lambda}_{\pm}= \pm \sqrt{(1-\frac{2}{\lambda})} x + 1$ (the case $0<\lambda< \frac{1}{2}$ can be treated similarly).

In general terms, let $(\Delta, {\bf L}) \subset \tor^*$ be a labelled polytope associated to a toric symplectic manifold $(M,\omega)$ while $f>0$ is a positive affine function and $u \in S(\Delta, {\bf L})$ is a solution of \eqref{modified-abreu} defining a compatible K\"ahler metric $g^u$ on $M$ with $(1/f^2) g^u$ Einstein--Maxwell. Let $F$ be an affine transformation of $\tor^*$ preserving $(\Delta, {\bf L})$. Then it follows from  the theory developed in Section~\ref{s:toric} that the pull-back $\bar u :=F^* u$ is a symplectic potential in $\mathcal{S}(\Delta, {\bf L})$, which in turn is a solution of \eqref{modified-abreu} with respect to the positive affine function $\bar f :=F^* f$. In particular, $\bar u$ defines another
K\"ahler metric $g^{\bar u}$ compatible with $\omega$, such that $(1/{\bar f}^2) g^{\bar u}$ is also Einstein--Maxwell.

We now apply this in our case, with $F(x, y)= (-x, y)$ which clearly preserves the labelled square $(\Delta, {\bf L}^{\lambda})$ and sends $f^{\lambda}_+(x, y)$ to $f^{\lambda}_-(x,y)$. Let $u_+(x,y)$ be a symplectic potential corresponding to the solution of \eqref{modified-abreu} with respect to $f^{\lambda}_+(x,y)$ and $u_{-}(x,y):= u_{+}(-x,y)$ the solution of \eqref{modified-abreu} corresponding to $f^{\lambda}_-(x,y)$. As we have already noticed, the proof of Theorem~\ref{ambitoric-EM-classification} implies that the corresponding K\"ahler metric $g^{u_+}$ must be ambitoric of product type, i.e. is written as
$$g^{u_+} = \frac{dx^2}{A(x)} + A(x) dt_1^2  + \frac{dy^2}{B(y)} + B(y) dt_2^2.$$
By \cite{LeB0},  or  by the discussion in Sections~\ref{prod} and \ref{s:compactification}, $B(y)= \frac{1}{\lambda}(1- y^2)$ (as $B(\pm 1)=0, B'(\pm 1) = \mp \frac{2}{\lambda}$), whereas $A(x)$ is a degree $4$ polynomial satisfying the boundary conditions $A(\pm 1)=0, A'(\pm 1) = \mp 2$. The general form of such a polynomial is $$A(x) = (1- x^2) + c(1-x^2)^2, $$
where $c$ is a real constant.
It follows that $A(-x)=A(x)$, showing ${\rm Hess}(u_+) = {\rm Hess}(u_-)$, i.e. $g^{u_+} = g^{u_-}$. Let us denote this $\omega_{\lambda}$-compatible product K\"ahler metric by $g$. Thus, the two conformal Einstein--Maxwell metrics $\frac{1}{(f_{\pm}^{\lambda})^2} g$ are conformal to the same K\"ahler metric $g$.

In fact, the involution $F$ of $\tor ^*$ comes from the antipodal map (still denoted by $F$) on $M= S^2 \times S^2$ over the first factor, which sends the momentum variable $x$ to $-x$ and leaves $t_1$ unchanged. It follows from the above discussion that $F$ is an isometry  of $g$ which interchanges the two riemannian metrics $\frac{1}{(f_{\pm}^{\lambda})^2} g$. Notice, however,  that $F$ is not a biholomorphism with respect to the complex structure of $(g, \omega_{\lambda})$. \end{rem}

\end{document}